\documentclass[11pt]{amsart}

\usepackage{enumitem}
\usepackage[svgnames]{xcolor}
\usepackage[all]{xy}
\usepackage{framed}

\usepackage[normalem]{ulem}

\usepackage{verbatim}
\usepackage{graphicx,epstopdf}

\newcommand{\foot}[1]{\mbox{}\marginpar{\raggedleft\hspace{0pt}\tiny #1}}

\newcommand{\Pexp}{P_\mathrm{exp}^\perp}

\newcommand{\NN}{\mathbb{N}}
\newcommand{\ZZ}{\mathbb{Z}}
\newcommand{\RR}{\mathbb{R}}
\newcommand{\AAA}{\mathcal{A}}
\newcommand{\BBB}{\mathcal{B}}
\newcommand{\CCC}{\mathcal{C}}

\newcommand{\DDD}{\mathcal{D}}
\newcommand{\EEE}{\mathcal{E}}

\newcommand{\SSS}{\mathcal{S}}
\newcommand{\PPP}{\mathcal{P}}

\newcommand{\WWW}{\mathcal{W}}

\newcommand{\GGG}{\mathcal{G}}
\newcommand{\MMM}{\mathcal{M}}
\newcommand{\Mf}{\MMM_f}
\newcommand{\htop}{h_\mathrm{top}}
\newcommand{\ulim}{\varlimsup}

\newcommand{\eps}{\varepsilon}
\newcommand{\symdiff}{\bigtriangleup}
\newcommand{\ph}{\varphi}

\newcommand{\Lamb}{\Lambda}
\newcommand{\NE}{\mathrm{NE}}
\newcommand{\xx}{\mathbf{x}}

\newcommand{\ww}{\mathbf{w}}

\newcommand{\ttau}{\boldsymbol{\tau}}
\newcommand{\lll}{\boldsymbol{\ell}}
\newcommand{\altP}{[\PPP]}
\newcommand{\altS}{[\SSS]}

\newtheorem{theorem}{Theorem}[section]
\newtheorem{definition}[theorem]{Definition}
\newtheorem{lemma}[theorem]{Lemma}
\newtheorem{proposition}[theorem]{Proposition}
\newtheorem{corollary}[theorem]{Corollary}

\newtheorem{thma}{Theorem}

\theoremstyle{remark}
\newtheorem{remark}[theorem]{Remark}

\usepackage{ulem}

\DeclareMathOperator{\Leb}{Leb}

\DeclareMathOperator{\diam}{diam}
\DeclareMathOperator{\Var}{Var}

\numberwithin{equation}{section}

\begin{document}

\title[Unique equilibrium states]{Unique equilibrium states for flows and homeomorphisms with non-uniform structure}
\author{Vaughn Climenhaga}
\author{Daniel J. Thompson}
\address{Department of Mathematics, University of Houston, Houston, Texas 77204}
\address{Department of Mathematics, The Ohio State University, 100 Math Tower, 231 West 18th Avenue, Columbus, Ohio 43210}
\email{climenha@math.uh.edu}
\email{thompson@math.osu.edu}
\date{\today}
\thanks{V.C.\ is supported by NSF grant DMS-1362838.  D.T.\ is supported by NSF grants DMS-$1101576$ and DMS-$1461163$. We acknowledge the hospitality of the American Institute of Mathematics, where some of this work was completed as part of a  SQuaRE} 

\begin{abstract}
Using an approach due to Bowen, Franco showed that continuous expansive flows with specification have unique equilibrium states for potentials with the Bowen property. We show that this conclusion remains true using weaker non-uniform versions of specification, expansivity, and the Bowen property. We also establish a corresponding result for homeomorphisms. In the homeomorphism case, we obtain the upper bound from the level-2 large deviations principle for the unique equilibrium state. The theory presented in this paper provides the basis for an ongoing program to develop the thermodynamic formalism in partially hyperbolic and non-uniformly hyperbolic settings.
\end{abstract}
\maketitle

\normalem 
\section{Introduction}
Let $X$ be a compact metric space and $F=(f_t)_{t\in \RR}$ a continuous flow on $X$.  Given a potential function $\phi\colon X\to\RR$, we study the question of existence and uniqueness of equilibrium states for $(X,F,\phi)$ -- that is, invariant measures which maximize the quantity $h_\mu(f_1) + \int\phi\,d\mu$.  We also study the same question for homeomorphisms $f\colon X\to X$.  This problem has a long history \cite{rB74, rB75, HK82, DKU90, oS99, IT10, PSZ, Pa15,CP16} and is connected with the study of global statistical properties for dynamical systems \cite{dR76, yK90, PP90, BSS02, CRL11, vC15}. 

For homeomorphisms, Bowen showed \cite{rB74} that $(X,f,\phi)$ has a unique equilibrium state whenever $(X,f)$ is an expansive system with specification and $\phi$ satisfies a certain regularity condition (the Bowen property).  Bowen's method was adapted to flows by Franco \cite{eF77}. Previous work by the authors established similar uniqueness results for shift spaces with a broad class of potentials \cite{CT,CT2}, and for non-symbolic discrete-time systems in the case $\phi\equiv 0$ \cite{CT3}.  In this paper, we consider potential functions satisfying a non-uniform version of the Bowen property in both the discrete- and continuous-time case.  

While we do not explore applications of this theory in this paper, we emphasize that the results are developed with a view to novel applications in the setting of smooth dynamical systems beyond uniform hyperbolicity. In particular, the main theorems of this paper are applied to diffeomorphisms with weak forms of hyperbolicity in \cite{CFT} and to geodesic flows in non-positive curvature in \cite{BCFT}.

We review the main points of our techniques for proving uniqueness of equilibrium states for maps, referring the reader to \cite{CT2,CT3} for details. 
Our approach is based on weakening each of the three hypotheses  of Bowen's theorem: expansivity, the specification property, and regularity of the potential.
Instead of asking for specification and regularity to hold globally, we ask for these properties to hold on a suitable collection of orbit segments $\GGG$.  Instead of asking for expansivity to hold globally, we ask that all measures with large enough free energy should observe expansive behavior. 

These ideas lead naturally to a notion of orbit segments which are obstructions to specification and regularity, and measures which are obstructions to expansivity. The guiding principle of our approach is that if these obstructions have less topological pressure than the whole space, then a version of Bowen's strategy 
can still be developed.
Some of the main points are as follows:
\begin{enumerate}
\item 
For a discrete-time dynamical system, we work with $X\times \NN$, which we think of as the space of orbit segments by identifying $(x, n)$ with $(x, f(x), \ldots, f^{n-1}(x))$.  At the heart of our approach is the concept of a \emph{decomposition} $(\PPP, \GGG, \SSS)$ for $X \times \NN$.  
We ask for specification and regularity to hold on a  collection of `good' orbit segments $\GGG \subset X\times \NN$, while the collections of $\PPP, \SSS \subset X \times \NN$ are thought of as `bad' orbit segments which are obstructions to specification and regularity. We ask that any orbit segment can be decomposed as a `good core' that is preceded and succeeded by elements of $\PPP$ and $\SSS$, respectively.  More precisely, for any $(x,n)$, there are numbers $p,g,s \in \NN \cup \{0\}$ so that $p+g+s = n$, and 
\[
(x,p) \in \PPP,\quad (f^px, g) \in \GGG,\quad (f^{p+g}x, s) \in \SSS.
\]
The choice of the decomposition $(\PPP, \GGG, \SSS)$ depends on the setting of any given application, and the dynamics of the situation are encoded in this choice. 
\item We define a natural version of topological pressure for orbit segments, and we require that the topological pressure of $\PPP \cup \SSS$, which we think of as the pressure of the obstructions to specification and regularity, is less than that of the whole space.
\item The positive expansivity property introduced in \cite{CT3} is that for small $\eps$, $\Gamma_\eps(x) = \{y: d(f^nx, f^ny) 
\leq \eps \mbox{ for all } n \geq0\} =\{x\}$
for $\mu$-almost every $x$, for any ergodic $\mu$ with $h_{\mu}(f)>h$, where $h$ is a constant less than $\htop(f)$. We think of the smallest $h\geq0$  so that this is true as the entropy of obstructions to expansivity.
\end{enumerate}


Under these hypotheses, our strategy is then inspired by Bowen's: his main idea was to construct an equilibrium state with the Gibbs property, and to show that this rules out the existence of a mutually singular equilibrium state. We obtain a certain Gibbs property which only applies to orbit segments in $\GGG$, and then we have to work to show that this is still sufficient to prove uniqueness of the equilibrium state.  

The above strategy was carried out in \cite{CT,CT2,CT3} under the assumption that either $(X,f)$ is a shift space or $\phi = 0$.  In this paper, we work in the setting of a continuous flow or homeomorphism on a compact metric space, and a continuous potential function.  This necessitates several new developments, which we now describe. For homeomorphisms and flows, we develop a theory for potential functions which are regular only on `good' orbit segments. The lack of global regularity introduces fundamental technical difficulties not present in the classical theory or the symbolic setting. For flows, which are the main focus of this paper, we work with the space $X\times \RR^+$, where the pair $(x,t)$ is thought of as the orbit segment $\{f_s(x)\mid 0\leq s< t\}$.   The main points addressed in this paper are:
\begin{enumerate}
\item Our potentials are not regular on the whole space, and this forces us to introduce and control non-standard `two-scale' partition sums throughout the proof (see \S\ref{sec:partition-sums}).
\item For flows, expansivity issues can be subtle and require new ideas beyond the discrete-time case. We introduce the notion of \emph{almost expansivity} for a flow-invariant ergodic measure (\S\ref{sec:almost-expansivity}), adapting a discrete-time version of this definition which was used in  \cite{CT3}. We also introduce the notion of \emph{almost entropy expansivity }(\S\ref{sec:aee0}) for 
a map-invariant ergodic measure. This is a natural 
analogue of entropy expansivity \cite{rB72}, adapted to apply to almost every point in the space.
Measures which are almost expansive for the flow are almost entropy expansive for the time-$t$ map. Almost entropy expansivity plays a crucial role in our proof via Theorem \ref{thm:aee}, a general ergodic theoretic result that strengthens \cite[Theorem 3.5]{rB72}.
\item Adapting the framework introduced in \cite{CT3} to the case of flows requires careful control of small differences in transition times, particularly in Lemma \ref{lem:multiplicity}.
\item The unique equilibrium state we construct admits a weak upper Gibbs bound, which in many cases we use to obtain
the upper bound from the level-2 large deviations principle, using results of Pfister and Sullivan (see \S\ref{sec:LDP}).
\end{enumerate}

We now state a version of our main result, which should be understood as a formalization of the strategy described previously.  We introduce our notation, referring the reader to \S \ref{sec:defs} for precise definitions:
$P(\phi)$ is the standard topological pressure; the quantity $\Pexp(\phi)$ is the largest free energy of an ergodic measure which observes non-expansive behavior; the specification property and Bowen property are versions of the classic properties which apply only on $\GGG$ rather than globally; the expression $P( \altP \cup \altS,\phi)$ is the topological pressure  of the obstructions to specification and regularity.
\begin{thma}\label{thm:flowssimple}
Let $(X,F)$ be a continuous flow on a compact metric space, and $\phi\colon X\to\RR$ a continuous potential function.  Suppose that $\Pexp(\phi) < P(\phi)$ and that $X\times \RR^+$ admits a decomposition $(\PPP, \GGG, \SSS)$ with the following properties: 
\begin{enumerate}[label=\textup{\textbf{(\Roman{*})}}]
\item\label{Wspec} $\GGG$ has the weak specification property;
\item\label{G-Bowen} $\phi$ has the Bowen property on $\GGG$;
\item\label{P-gap}
 $P( \altP \cup \altS,\phi) < P(\phi)$.
\end{enumerate}
Then $(X,F,\phi)$ has a unique equilibrium state.
\end{thma}
In fact, we will prove a slightly more general result, of which Theorem \ref{thm:flowssimple} is a corollary.
The more general version, Theorem \ref{thm:flowsD}, 
applies under slightly weaker versions of our hypotheses, which we discuss and motivate in \S\ref{sec:main-results-for-flows}.

We also develop versions of our results that apply for homeomorphisms. These discrete-time arguments are analogous to, and easier than, the flow case, so we just outline the proof, highlighting any differences with the flow case. Our main results for homeomorphisms are Theorem \ref{thm:mapssimple}, which is the analogue of Theorem \ref{thm:flowssimple}, and Theorem \ref{thm:mapsD}, which is the analogue for homeomorphisms of Theorem \ref{thm:flowsD}.  Finally, in Theorem \ref{thm:ldp}, we establish the upper level-2 large deviations principle for the unique equilibrium states provided by Theorem \ref{thm:mapssimple}.

\subsection*{Structure of the paper} 

We collect our definitions, particularly for flows, in \S \ref{sec:defs}. Our main results for flows are proved in \S\S\ref{sec:exp}--\ref{sec:flows-pf}. Our main results for maps are proved in \S\S\ref{sec:maps}--\ref{sec:maps-pf}. 
In \S\ref{sec:LDP}, we prove the large deviations results of Theorem \ref{thm:ldp}. In \S\ref{sec:aee}, we prove Theorem \ref{thm:aee}, which is a self-contained result about measure-theoretic entropy for almost entropy expansive measures.
\section{Definitions}\label{sec:defs}

In this section we give the relevant definitions for flows; the corresponding definitions for maps are given in \S\ref{sec:maps}.

\subsection{Partition sums and topological pressure}\label{sec:partition-sums}

Throughout, $X$ will denote a compact metric space and $F=(f_t)_{t\in \RR}$ will denote a continuous flow on $X$.  We write $\MMM_F(X)$ for the set of Borel $F$-invariant probability measures on $X$.  Given $t\geq 0$, $\delta>0$, and $x,y\in X$ we define the \emph{Bowen metric}
\begin{equation}\label{eqn:dt}
d_t(x,y) := \sup \{ d(f_sx, f_sy) \mid s\in [0,t]\},
\end{equation}
and the \emph{Bowen balls}
\begin{equation}\label{eqn:Bowen-ball}
\begin{aligned}
B_t(x,\delta) &:= \{y\in X \mid d_t(x,y) < \delta\}, \\
\overline B_t(x,\delta) &:= \{y\in X \mid d_t(x,y) \leq \delta\}.
\end{aligned}
\end{equation}
Given $\delta>0$, $t\in \RR^+$, and $E\subset X$, we say that $E$ is $(t,\delta)$-separated if for every distinct $x,y\in E$ we have $y\notin \overline B_t(x,\delta)$.  Writing $\RR^+ = [0, \infty)$, we view $X\times \RR^+$ as the space of finite orbit segments for $(X,F)$ by associating to each pair $(x,t)$ the orbit segment $\{f_s(x) \mid 0\leq s< t\}$. Our convention is that $(x,0)$ is identified with the empty set rather than the point $x$.  Given $\CCC \subset X\times \RR^+$ and $t\geq0$ we write $\CCC_t = \{x\in X \mid (x,t) \in \CCC\}$.  

Now we fix a continuous potential function $\phi\colon X\to \RR$.  Given a fixed scale $\eps>0$, we use $\phi$ to assign a weight to every finite orbit segment by putting
\begin{equation}\label{eqn:Phi}
\Phi_\eps(x,t) = \sup_{y\in B_t(x,\eps)} \int_0^t \phi(f_sy)\,ds.
\end{equation}
In particular, $\Phi_0(x,t) = \int_0^t \phi(f_sx)\,ds$. The general relationship between $\Phi_\eps(x,t)$ and $\Phi_0(x,t)$ is that
\begin{equation}\label{eqn:Phi0eps}
|\Phi_\eps(x,t) - \Phi_0(x,t)| \leq t\Var(\phi,\eps),
\end{equation}
where $\Var(\phi,\eps) = \sup \{|\phi(x)-\phi(y)| \mid d(x,y)<\eps\}$.

Given $\CCC \subset X\times \RR^+$ and $t>0$, we consider the \emph{partition function}
\begin{equation}\label{eqn:partition-sum}
\Lambda(\CCC,\phi,\delta,\eps,t) = \sup \left\{ \sum_{x\in E} e^{\Phi_{\eps}(x,t)} \mid E\subset \CCC_t \text{ is $(t,\delta)$-separated} \right\}.
\end{equation}
We will often suppress the function $\phi$ from the notation, since it is fixed throughout the paper, and simply write $\Lambda(\CCC,\delta,\eps,t)$. 
When $\CCC = X\times \RR^+$ is the entire system, we will simply write $\Lambda(X,\phi,\delta,\eps,t)$ or $\Lambda(X, \delta,\eps,t)$. We call a $(t, \delta)$-separated set that attains the supremum in \eqref{eqn:partition-sum} \emph{maximizing} for $\Lambda(\CCC,\delta,\eps,t)$.  We are only guaranteed the existence of such sets when $\CCC = X\times \RR^+$, since otherwise $\CCC_t$ may not be compact.

The pressure of $\phi$ on $\CCC$ at scales $\delta,\eps$ is given by
\begin{equation} \label{def:pressure}
P(\CCC,\phi,\delta,\eps) = \ulim_{t\to\infty} \frac 1t \log \Lambda(\CCC,\phi,\delta,\eps,t).
\end{equation}
Note that $\Lambda$ is monotonic in both $\delta$ and $\eps$, but in different directions; thus the same is true of $P$.  Again, we write $P(\CCC,\phi,\delta)$ in place of $ P(\CCC,\phi,\delta,0)$ to agree with more standard notation, and we let  
\begin{equation} \label{def:pressure0}
P(\CCC,\phi) = \lim_{\delta \to 0} P(\CCC,\phi,\delta).
\end{equation}
When $\CCC = X\times \RR^+$ is the entire space of orbit segments, the topological pressure reduces to the usual notion of topological pressure on the entire system, and we write $P(\phi, \delta)$ in place of $P(\CCC, \phi, \delta)$, and $P(\phi)$ in place of $P(\CCC, \phi)$.  The variational principle for flows \cite{BR75} states that $P(\phi) = \sup \{ h_\mu(f_1) + \int\phi\,d\mu \mid \mu\in \MMM_F(X)\}$, where $h_\mu(f_1)$ is the usual measure-theoretic entropy of the time-$1$ map of the flow. A measure achieving the supremum is called an \emph{equilibrium state}.  

\begin{remark}\label{rmk:two-scale-b}
The most obvious definition of partition function would be to take $\eps=0$ so that the weight given to each orbit segment is determined by the integral of the potential function along that exact orbit segment, rather than by nearby ones. To match more standard notation, we often write $\Lambda(\CCC,\phi,\delta,t)$ in place of $ \Lambda(\CCC,\phi,\delta, 0, t)$.  The partition sums $\Lambda(\CCC,\phi,\delta,\eps,t)$ arise throughout this paper, particularly in \S \ref{sec:lower-X} and \S \ref{sec:adapted}.   The relationship between the two quantities can be summarised as follows.
\begin{enumerate}
\item If $(X,F)$ is expansive at scale $\eps$, then $P(\CCC,\phi,\delta,\eps) = P(\CCC,\phi,\delta,0)$.
\item If $\phi$ is Bowen at scale $\eps$, then the two pressures above are equal, and moreover the ratio between $\Lambda(\CCC,\phi,\delta,\eps,t)$ and $\Lambda(\CCC,\phi,\delta,0,t)$ is bounded away from $0$ and $\infty$.
\item In the absence of regularity or expansivity assumptions, we have the relationship
\[
e^{-t\Var(\phi, \eps)} \Lambda(\CCC,\phi,\delta,\eps,t) \leq \Lambda(\CCC,\phi, \delta, t) \leq e^{t\Var(\phi, \eps)} \Lambda(\CCC,\phi,\delta,\eps,t),
\]
and thus $|P(\CCC,\phi,\delta,\eps) - P(\CCC,\phi,\delta)| \leq \Var(\phi,\eps)$. By continuity of $\phi$, this establishes that $P(\CCC,\phi,\delta,\eps) \to P(\CCC,\phi,\delta,0)$ as $\eps\to 0$, but does not give us the conclusions of (1) or (2). 
\end{enumerate}
Because our versions of expansivity and the Bowen property do not hold globally, we are in case (3) above, so a priori we cannot replace $\Lambda(\CCC,\phi,\delta,\eps,t)$ with $\Lambda(\CCC,\phi,\delta,t)$ in the proofs.
\end{remark}

\begin{remark}
We can restrict to $(t, \delta)$-separated sets of maximal cardinality $E \subset \CCC_t$ in the definition of pressure: these always exist, even when $\CCC_t$ is non-compact, since the possible values for the cardinality are finite (by compactness of $X$). If $E$ were not of maximal cardinality, we could just add in another point, which would increase the partition sum \eqref{eqn:partition-sum}. Furthermore, a $(t,\delta)$-separated set $E$ of maximal cardinality is $(t,\delta)$-spanning in the sense that $D_t \subset \bigcup_{x\in E} \overline B_t(x,\delta)$.  If this were not so then we could add another point to $E$ and increase the cardinality.
\end{remark}

\subsection{Decompositions}
We introduce the notion of a decomposition for a sub-collection of the space of orbit segments.
\begin{definition} \label{def:decompflow}
A \emph{decomposition} $(\PPP, \GGG, \SSS)$ for $\DDD \subseteq X \times \RR^+$ consists of three collections $\mathcal{P}, \mathcal{G}, \mathcal{S}\subset X\times \RR^+$ and three functions $p,g,s : \DDD \to \RR^+$ such that for every $(x,t)\in \DDD$, the values $p=p(x,t)$, $g=g(x,t)$, and $s=s(x,t)$ satisfy $t = p+g+s$, and 
\begin{equation}\label{eqn:decomposition}
(x,p)\in \mathcal{P}, \quad (f_px, g)\in\mathcal{G}, \quad (f_{p+g}x, s)\in \mathcal{S}.
\end{equation}
If $\DDD = X \times \RR^+$, we say that $(\PPP, \GGG, \SSS)$ is a decomposition for $(X, F)$.
Given a decomposition $(\PPP,\GGG,\SSS)$ and $M\in \RR^+$, we write $\GGG^M$ for the set of orbit segments $(x,t) \in \DDD$ for which $p \leq M$ and $s\leq M$.
\end{definition}
We make a standing assumption that $X \times \{0 \} \subset \PPP \cap \GGG \cap \SSS$ to allow for orbit segments to be decomposed in `trivial' ways; for example, $(x,t)$ can belong `purely' to one of the collections $\PPP$, $\GGG$, or  $\SSS$ or can transition directly from $\PPP$ to $\SSS$ -- note that formally the symbols $(x,0)$ are identified with the empty set. This is implicit in our earlier work \cite{CT,CT2,CT3}. 

We will be interested in decompositions where $\GGG$ has specification, $\phi$ has the Bowen property on $\GGG$, and $\PPP \cup \SSS$ carries smaller pressure than the entire system.  
In the case of flows, a priori we must replace the collections $\PPP$ and $\SSS$ that appear in the decomposition with a related and slightly larger collection $[\PPP] \cup [\SSS]$, where given $\CCC\subset X\times \RR^+$ we write
\begin{equation}\label{eqn:C1}
[\CCC] := \{(x,n) \in X\times \NN \mid (f_{-s}x, n+s+t) \in \CCC \text{ for some }s,t\in [0,1]\}.
\end{equation}
Passing from $\PPP \cup \SSS$ to $[\PPP] \cup [\SSS]$ ensures that  the decomposition is well behaved with respect to replacing continuous time with discrete time. This issue occurs in  Lemma \ref{lem:many-in-G}. 

\begin{figure}[htbp]
\includegraphics[width=.95\textwidth]{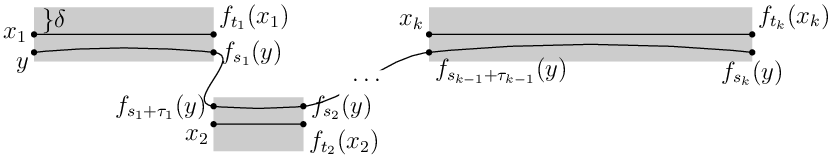}
\caption{The specification property.}
\label{fig:specification}
\end{figure}

\subsection{Specification} \label{sec:specf}
We say that $\GGG\subset X\times \RR^+$ has \emph{weak specification at scale $\delta$} if there exists $\tau>0$ such that for every $\{(x_i,t_i)\}_{i=1}^k \subset \GGG$ there exists a point $y$ and a sequence of ``gluing times'' $\tau_1,\dots,\tau_{k-1} \in \RR^+$ with $\tau_i \leq \tau$ such that for $s_j = \sum_{i=1}^{j} t_i + \sum_{i=1}^{j-1}\tau_i$ and $s_0 = \tau_0 = 0$, we have (see Figure \ref{fig:specification})
\begin{equation}\label{eqn:spec}
d_{t_j}(f_{s_{j-1}+\tau_{j-1}}y, x_j) < \delta \text{ for every } 1\leq j\leq k.
\end{equation}
We say that $\GGG\subset X\times \RR^+$ has \emph{weak specification at scale $\delta$ with maximum gap size $\tau$} if we want to declare a value of $\tau$ that plays the role described above. We say that $\GGG\subset X\times \RR^+$ has \emph{weak specification} if it has weak specification at every scale $\delta>0$. 
\begin{remark}
We often write \emph{(W)-specification} as an abbreviation for weak specification. Furthermore, since \emph{(W)-specification} is the only version of the specification property considered in this paper, we henceforth use the term \emph{specification} as shorthand for this property.
\end{remark}
Intuitively, \eqref{eqn:spec} means that there is some point $y$ whose orbit shadows the orbit of $x_1$ for time $t_1$, then after a ``gap'' of length at most $\tau$, shadows the orbit of $x_2$ for time $t_2$, and so on.
Note that $s_j$ is the time spent for the orbit $y$ to shadow the orbit segments $(x_1, t_1)$ up to $(x_j, t_j)$. Note that we differ from Franco \cite{eF77} in allowing $\tau_j$ to take any value in $[0,\tau]$, not just one that is close to $\tau$.  This difference is analogous in the discrete time case to the difference between (S)-specification where we take the transition times exactly $\tau$, or (W)-specification where the transition times are bounded above by $\tau$. 
Franco also asks that the shadowing orbit $y$ can be taken to be periodic, and that the gluing time $\tau_j$ does not depend on any of the orbit segments $(x_i, t_i)$ with $i>j+1$.

We can weaken the definition of specification so that it only applies to elements of $\GGG$ that are sufficiently long. This gives us some useful additional flexibility which we exploit in Lemma \ref{lem:2-specs}.

\begin{definition}
We say that $\GGG\subset X\times \RR^+$ has \emph{tail (W)-specification at scale $\delta$} if there exists $T_0>0$ so that $\GGG \cap (X\times [T_0,\infty))$ has weak specification at scale $\delta$; i.e.\ the specification property holds for the collection of orbit segments
\[
\{(x_i, t_i) \in \GGG \mid t_i\geq T_0 \}.
\]
We also sometimes write ``$\GGG$ has (W)-specification at scale $\delta$ for $t\geq T_0$'' to describe this property. 
\end{definition}

\subsection{The Bowen property} The Bowen property was first defined for maps in \cite{rB74}, and extended to flows by Franco \cite{eF77}. We give a version of this definition for a collection of orbit segments $\CCC$. 
\begin{definition} \label{def:Bowen}
Given $\CCC\subset X\times \RR^+$, a potential $\phi$  has the \emph{Bowen property on $\CCC$ at scale $\eps$} if there exists $K>0$ so that
\begin{equation}\label{eqn:Bowen}
\sup \{|\Phi_0(x,t) - \Phi_0(y,t)| : (x,t) \in \CCC, y \in B_t(x, \eps) \} \leq K.
\end{equation}
We say $\phi$ has the \emph{Bowen property on $\CCC$} if there exists $\eps>0$ so that $\phi$ has the Bowen property on $\CCC$ at scale $\eps$.
\end{definition}

In particular, we say that $\phi\colon X\to \RR$ has the \emph{Bowen property} if $\phi$ has the Bowen property on $\CCC = X\times \RR^+$; this agrees with the original definition of Bowen and Franco. This dynamically-defined regularity property is central to Bowen's proof of uniqueness of equilibrium states. For a uniformly hyperbolic system, every H\"older potential $\phi$ has the Bowen property. This is no longer true in non-uniform hyperbolicity; for example, the geometric potential $-\log f'$ for the Manneville--Pomeau map $f(x)= x+x^{1+\alpha}$ 
is a natural potential which is H\"older but not Bowen. Asking for the Bowen property to hold on a collection $\GGG$ rather than globally allows us to deal with non-uniformly hyperbolic systems where one only expects this kind of regularity to hold for those orbit segments which experience a definite amount of hyperbolicity, and where it may not be known whether natural potentials such as the geometric potential are H\"older \cite{CFT,BCFT}.

We sometimes call $K$ the \emph{distortion constant} for the Bowen property. Note that if $\phi$ has the Bowen property at scale $\eps$ on $\GGG$ with distortion constant $K$, then for any $M>0$, $\phi$ has the Bowen property at scale $\eps$ on $\GGG^M$ with distortion constant given by $K(M) = K+2M \Var(\phi, \eps)$. 


\subsection{Almost expansivity}\label{sec:almost-expansivity}

Given $x\in X$ and $\eps>0$, consider the set
\begin{equation}\label{eqn:Binfty}
\Gamma_\eps(x) := \{y\in X \mid d(f_tx,f_ty) \leq \eps \text{ for all }t\in \RR \},
\end{equation}
which can be thought of as a two-sided Bowen ball of infinite order for the flow.  
Note that $\Gamma_\eps(x) = \bigcap_{t\in \RR} f_t \overline B_{2t}(x,\eps)$ is compact for every $x,\eps$.

Expansivity for flows was defined by Bowen and Walters; their definition, details of which can be found in \cite{BW72}, implies that for every $s>0$, there exists $\eps>0$ such that
\begin{equation}\label{eqn:expansive-orbit}
\Gamma_\eps(x) \subset f_{[-s,s]}(x) := \{f_t(x) \mid t\in [-s,s]\}
\end{equation}
for every $x\in X$. Since points on a small segment of orbit always stay close for all time, \eqref{eqn:expansive-orbit} essentially says that the set $\Gamma_\eps(x)$ is the smallest possible. Thus, we declare the set of non-expansive points to be those where \eqref{eqn:expansive-orbit} fails. We want to consider measures that witness expansive behaviour, so we declare an almost expansive measure to be one that gives zero measure to the non-expansive points. This is the content of the next definition.
\begin{definition}
 Given $\eps>0$, the set of \emph{non-expansive points at scale $\eps$} for a flow $(X,F)$  is the set
\begin{align*}
\NE(\eps):=\{ x\in X \mid \Gamma_\eps(x)\not\subset  f_{[-s,s]}(x) \text{ for any }s>0 \}.
\end{align*}
We say that an $F$-invariant measure $\mu$ is \emph{almost expansive at scale $\eps$} if $\mu(\mathrm{NE}(\eps))=0$.  
\end{definition}
\
A measure $\mu$ which is almost expansive at scale $\eps$ gives full measure to the set of points $x$ for which there exists $s=s(x)$ for which \eqref{eqn:expansive-orbit} holds. 
We remark that in contrast to the Bowen-Walters definition, we allow $s(x)$ to be large or even unbounded. Furthermore, our hypotheses do not preclude the existence of fixed points for the flow; for expansive flows, fixed points can only be isolated [BW72, Lemma 1] and can hence be disregarded.

The following definition gives a quantity which captures the largest possible free energy of a non-expansive ergodic measure.
\begin{definition}\label{def:expansivity}
Given a potential $\phi$, the \emph{pressure of obstructions to expansivity at scale $\eps$} is 
\begin{align*}
\Pexp(\phi, \eps) &=\sup_{\mu\in \mathcal{M}_F^e(X)}\left\{ h_{\mu}(f_1) + \int \phi\, d\mu \mid \mu(\mathrm{NE}(\eps))>0\right\} \\
&=\sup_{\mu\in \mathcal{M}_F^e(X)}\left\{ h_{\mu}(f_1) + \int \phi\, d\mu \mid  \mu(\mathrm{NE}(\eps))=1\right\}.
\end{align*}
We define a scale-free quantity by
\[
\Pexp(\phi) = \lim_{\eps \to 0} \Pexp(\phi, \eps).
\]
\end{definition}
Note that $\Pexp(\phi,\eps)$ is non-increasing as $\eps \to 0$, which is why the limit in the above definition exists. It is essential that the measures in the first supremum are ergodic. If we took this supremum over invariant measures, and a non-expansive measure existed, we would include measures that are a convex combination of a non-expansive measure and a measure with large free energy,  so the supremum would equal the topological pressure.


\subsection{Main results for flows}\label{sec:main-results-for-flows}

Theorem \ref{thm:flowssimple} will be deduced from the following more general result, which is proved in \S\S\ref{sec:exp}--\ref{sec:flows-pf}.
\begin{theorem}\label{thm:flowsD}
Let $(X,F)$ be a continuous flow on a compact metric space, and $\phi\colon X\to\RR$ a continuous potential function.  Suppose there are $\delta,\eps>0$ with $\eps > 40\delta$ such that $\Pexp(\phi,\eps) < P(\phi)$ and there exists $ \DDD \subset X\times \RR^+$ which admits a decomposition $(\PPP, \GGG, \SSS)$ with the following properties:
\begin{enumerate}[label=\textup{\textbf{(\Roman{*}$'$)}}]
\item\label{specD} For every $M\in \RR^+$, $\GGG^M$ has tail (W)-specification at scale $\delta$; 
\item\label{Bowen2} 
$\phi$ has the Bowen property at scale $\eps$ on $\GGG$;
\setcounter{enumi}{2}
\item\label{gapD}
$P(\DDD^c \cup \altP \cup \altS, \phi,\delta, \eps) < P(\phi)$.
\end{enumerate}
Then $(X,F,\phi)$ has a unique equilibrium state.
\end{theorem}

These hypotheses weaken those of Theorem \ref{thm:flowssimple} in two main directions. 
\begin{enumerate}
\item Theorem \ref{thm:flowssimple} requires that every orbit segment has a decomposition, while Theorem \ref{thm:flowsD} permits a set of orbit segments $\DDD^c \subset X\times \RR^+$ to have no decomposition, 
provided they carry less pressure than the whole system.
\item The hypotheses of Theorem \ref{thm:flowssimple} require knowledge of the system at all scales: in particular, the specification condition \ref{Wspec} in Theorem \ref{thm:flowssimple} requires specification to hold at \emph{every} scale $\delta>0$. Here, we require a specification property to be verified only at a fixed scale $\delta$, and all other hypotheses to be verified at a larger fixed scale $\eps$. An example where this is useful is the Bonatti--Viana family of diffeomorphisms, where in \cite{CFT} we are able to verify the discrete-time version of these hypotheses at suitably chosen scales, but establishing them for arbitrarily small scales is difficult, and perhaps impossible.

\end{enumerate}

We make a few more remarks on these hypotheses. By Remark \ref{rmk:two-scale-b}, we can guarantee \ref{gapD} by checking the bound
\begin{equation}\label{eqn:rough-gap}
P(\DDD^c \cup \altP \cup \altS, \phi,\delta) + \Var(\phi,\eps) < P(\phi).
\end{equation}

We do not claim that the relationship $\eps >40 \delta$ is sharp, but we do not expect that it can be significantly improved using these methods. The number $40$ does not have any special significance 
but it is unavoidable that we control the Bowen property and expansivity at a larger scale than where specification is assumed. 

If we assume the hypotheses of Theorem \ref{thm:flowssimple}, we can verify the hypotheses of Theorem \ref{thm:flowsD} by taking $\DDD=X \times \RR^+$, and any suitably small  $\delta,\eps>0$ with $\eps > 40\delta$. The only hypothesis which is not immediate to verify from the hypotheses of Theorem \ref{thm:flowssimple} is \ref{specD}, and this is verified by the following lemma.



\begin{lemma}\label{lem:2-specs}
Suppose that $\GGG \subset X\times \RR^+$ has tail specification at all scales $\delta>0$, then so does $\GGG^M$ for every $M>0$. In particular, \ref{Wspec} implies \ref{specD}.
\end{lemma}
\begin{proof}
Given $M>0$, let $\delta' = \delta'(M)>0$ be such that $d(x,y)<\delta'$ implies that $d(f_tx,f_ty) < \delta$ for every $t\in [0,M]$.  (Positivity of $\delta'$ follows from continuity of the flow and compactness of $X$.) Now let $T_0>0$ be such that $\GGG \cap (X\times [T_0,\infty))$ has specification at scale $\delta'$.  Given any $(x,t)\in \GGG^M$ with $t\geq T_0 + 2M$, we must have $g(x,t) \geq T_0$.  Thus if $(x_1,t_1),\dots, (x_k,t_k)$ is any collection of orbit segments in $\GGG^M$ with $t_i \geq T_0 + 2M$, then there are $s_i \in [0,M]$ and $t_i' \in [t_i - 2M, t_i]$ such that $\{(f^{s_i}x_i,t_i') \mid 1\leq i\leq k\}\subset \GGG$.  Since $t_i' \geq T_0$ we can use the specification property on $\GGG$ to get an orbit that shadows each $(f^{s_i}x_i,t_i')$ to within $\delta'$ (with transition times at most $\tau = \tau(\delta')$).  By our choice of $\delta'$, this orbit shadows each $(x_i,t_i)$ to within $\delta$ (with transition times at most $\tau$).  We conclude that $\GGG^M$ has tail specification at scale $\delta$.
\end{proof}

We conclude that Theorem \ref{thm:flowssimple} is a corollary of Theorem \ref{thm:flowsD}, and we now turn our attention to proving this more general statement.

\section{Weak expansivity and generating for adapted partitions}\label{sec:exp}

In this section, we develop some general preparatory results on generating properties of partitions in the presence of weak expansivity properties. 

\subsection{Almost entropy expansivity} \label{sec:aee0}

It is well known that the time-$t$ map of an expansive flow is entropy expansive. We develop an analogue of entropy expansivity for measures called \emph{almost entropy expansivity}, which has the property that if $\mu$ is almost expansive for a flow, then it is almost entropy expansive for the time-$t$ map of the flow. This property plays an important role in our proof, as entropy expansivity does for Franco, and is key to obtaining a number of results on generating for partitions. 

Let $X$ be a compact metric space and $f\colon X\to X$ a homeomorphism.  Let $\mu$ be an ergodic $f$-invariant Borel probability measure. For a set $K \subset X$, let $h(K)$ denote the  (upper capacity) entropy of $K$. That is, $h(K)$ corresponds to $P(\CCC,0)$ for $\CCC= K \times \NN$ as defined in \S\ref{sec:maps}, which is the natural analogue for maps of \eqref{eqn:partition-sum}--\eqref{def:pressure0}.

Given $x\in X$, consider the set
\begin{equation}\label{eqn:Binfty-fd}
\Gamma_\eps(x;f, d) := \{y\in X \mid d(f^n x,f^ny) \leq \eps \text{ for all }n\in \ZZ \}.
\end{equation}
Recall from \cite{rB72} that the map $f$ is said to be \emph{entropy expansive} if $h(\Gamma_\eps(x; f,d)=0$ for every $x\in X$.  We will need the following weaker notion.
\begin{definition}
We say that $\mu$ is \emph{almost entropy expansive} at scale $\eps$ (in the metric $d$) with respect to $f$ if $h(\Gamma_\eps(x ; f, d)) = 0$ for $\mu$-a.e. $x \in X$.
\end{definition} 
Our notation emphasizes the role of the metric $d$ because later in the paper we will need to use this notion relative to various metrics $d_t$.  Bowen proved that if $f$ is entropy expansive at scale $\eps$, then every partition $\AAA$ with diameter smaller than $\eps$ has $h_\mu(f,\AAA) = h_\mu(f)$.  This result was obtained as an immediate consequence of the main part of \cite[Theorem 3.5]{rB72}, which shows that for any $\eps>0$ and any partition with $\diam \AAA \leq \eps$, we have
\begin{equation}\label{eqn:bowen-h}
h_\mu(f) \leq h_\mu(f,\AAA) + \sup_{x\in X} h(\Gamma_\eps(x; f,d)).
\end{equation}
Clearly, $f$ is entropy expansive if the supremum is 0.  Similarly, one sees immediately that $\mu$ is almost entropy expansive at scale $\eps$ if and only if the \emph{essential} supremum
\begin{equation}\label{eqn:h*}
h^*(\mu,\eps; f, d) = \sup \{ \bar h \in \RR \mid \mu\{x \mid h(\Gamma_\eps(x ; f, d)) > \bar h\} > 0 \}
\end{equation}
vanishes, and we strengthen Bowen's result by showing that one can use the $\mu$-essential supremum in \eqref{eqn:bowen-h}.  The following theorem is proved in \S\ref{sec:aee}.


\begin{theorem}\label{thm:aee}
Let $X$ be a compact metric space and $f\colon X\to X$ a homeomorphism.  Let $\mu$ be an ergodic $f$-invariant Borel probability measure. If $\AAA$ is any partition with $\diam \AAA \leq \eps$ in the metric $d$, then
\begin{equation}\label{eqn:aee}
h_\mu(f) \leq h_\mu(f,\AAA) + h^*(\mu,\eps; f, d).
\end{equation}
In particular, if $\mu$ is almost entropy expansive at scale $\eps$, then every partition with diameter smaller than $\eps$ has $h_\mu(f) = h_\mu(f,\AAA)$.
\end{theorem}

To apply Theorem \ref{thm:aee} in the setting of our main results, we first relate  almost expansivity for the flow with almost entropy expansivity for the time-$t$ map of the flow. 
\begin{proposition}\label{prop:ae-generating0}
If $\mu\in \MMM_F(X)$ is almost expansive at scale $\eps$, then $\mu$ is almost entropy expansive (at scale $\eps$ in the metric $d_t$) with respect to the time-$t$ map $f_t$. 
\end{proposition}
\begin{proof}

It is immediate from the definitions that $\Gamma_\eps(x)= \Gamma_\eps(x; f_t, d_t)$. Thus, if $\mu$ is almost expansive for $F$, then
for $\mu$-a.e.\ $x$, the set $\Gamma_\eps(x; f_t, d_t)$ is contained in $f_{[-s,s]}(x)$ for some $s = s(x)\in \RR^+$. Fix such an $x$ and let $s =s(x)$. In what follows, we will show that $h(f_{[-s,s]}(x), f_t) = 0$. This shows that $h(\Gamma_\eps(x; f_t, d_t), f_t) = 0$, and since this argument applies to $\mu$-almost every $x$, it follows that $\mu$ is almost entropy expansive for $f_t$.

So, it just remains to show that the entropy of the finite orbit segment $f_{[-s,s]}(x)$ is $0$ with respect to $f_t$. Let $r>0$ be sufficiently small that $f_{[-r,r]}(y) \subset B(y,\eps)$ for all $y\in X$ (this is possible by continuity of the flow and compactness of the space). Given $\delta>0$, fix $N\in \NN$ large enough such that $s/N< r$.  Let $A = \{f_{kr}(x) \mid k=-N,\dots, N\}$, and note that $f_t(f_{[(k-1)r, (k+1)r]}(x)) \subset B(f_{t+kr}x, \eps)$ for all $t\in \RR$ and all $k$.  Thus, for every $n$, the set $A$ is $(nt,\delta)$-spanning under $F$ for $f_{[-s,s]}(x)$. It follows that, in the metric $d_t$, $A$ is $(n,\delta)$-spanning under $f_t$, which gives $h(\Gamma_\eps(x; f_t, d_t), f_t) = 0$.
\end{proof}
The following proposition, which plays a similar role as \cite[Proposition 2.6]{CT3},
is a consequence of Theorem \ref{thm:aee} and Proposition \ref{prop:ae-generating0}.

\begin{proposition}\label{prop:ae-generating}
If $\mu\in \MMM_F(X)$ is almost expansive at scale $\eps$ and $\AAA$ is a finite measurable partition of $X$ with diameter less than $\eps$ in the $d_t$ metric for some $t>0$, then the time-t map $f_t$ satisfies $h_\mu(f_t,\AAA) = h_\mu(f_t)$. 
\end{proposition}

\subsection{Adapted partitions and results on generating} \label{sec:ap}

We extend Proposition \ref{prop:ae-generating} to some useful results on generating using the notion of an adapted partition. This terminology was introduced in \cite{CT3}, although the concept goes back to Bowen \cite{rB73}.

\begin{definition}
Let $E_t$ be a $(t,\gamma)$-separated set of maximal cardinality.  A partition $\AAA_t$ of $X$ is \emph{adapted} to $E_t$ if for every $w\in \AAA_t$ there is $x\in E_t$ such that $B_t(x,\gamma/2) \subset w \subset \overline B_t(x,\gamma)$.
\end{definition}
Adapted partitions exist for any $(t,\gamma)$-separated set 
of maximal cardinality since the sets $B_t(x,\gamma/2)$ are disjoint and the sets $\overline B_t(x, \gamma)$ cover $X$.
\begin{lemma} \label{lem:genAt}
If $\mu\in \MMM_F(X)$ is almost expansive at scale $\eps$, and $\AAA_t$ is an adapted partition for a $(t,\eps/2)$-separated set $E_t$ of maximal cardinality, then $h_\mu(f_t,\AAA_t) = h_\mu(f_t)$.
\end{lemma}
\begin{proof}
For any $w \in \AAA_t$, there exists $x$ so that $w \subset \overline B_t(x,\eps/2)$; this shows that $\diam \AAA_t \leq \eps$ in the metric $d_t$. By Proposition \ref{prop:ae-generating}, we have $h_\mu(f_t,\AAA_t) = h_\mu(f_t)$.  \end{proof}
The proof of the following proposition requires both Lemma \ref{lem:genAt} and a careful use of the almost expansivity property to take a crucial step of replacing a term of the form $\Phi_\eps(x,t)$ with $\Phi_0(x,t)$.
\begin{proposition}\label{prop:all-pressure}
If $\Pexp(\phi,\eps) < P(\phi)$, then $P(\phi,\gamma)=P(\phi)$ for every $\gamma \in (0,\eps/2]$. 
\end{proposition}
\begin{proof}
Given an ergodic $\mu$, write $P_\mu(\phi) := h_\mu(f_1) + \int\phi\,d\mu$ for convenience.  We prove the proposition by showing that $P(\phi,\eps) \geq P_\mu(\phi)$ for every ergodic $\mu$ with $P_\mu(\phi) > \Pexp(\phi,\eps)$.  We do this by relating both $P(\phi,\eps)$ and $P_\mu(\phi)$ to an adapted partition.  In order to carry this out we first introduce a technical lemma that will be used both here 
and in the proof of Lemma \ref{lem:pos-for-es}.

Given a finite partition $\AAA$ and an $F$-invariant measure $\mu$, for each $w\in \AAA$ with $\mu(w)>0$ we define a function $\Phi = \Phi_\mu \colon \AAA\to \RR$ by
\begin{equation}\label{eqn:Phiw}
\Phi(w) := \frac 1{\mu(w)} \int_w \Phi_0(x,t) \,d\mu.
\end{equation}
Given $\alpha \in (0,1)$, write $H(\alpha) := -\alpha\log\alpha -(1-\alpha)\log(1-\alpha)$.

\begin{lemma}\label{lem:AD}
Suppose $\mu\in \MMM_F(X)$ is almost expansive at scale $\eps$, and let $\gamma \in (0,\eps/2]$. Let $\AAA_t$ be an adapted partition for a maximizing $(t,\gamma)$-separated set for $\Lambda(X, \gamma, t)$. Let $D\subset X$ be a union of elements of $\AAA_t$.  Then for every $t\in \RR^+$ we have
\[
t\left(h_\mu(f_1) + \int\phi\,d\mu\right)
\leq
\mu(D) \log \sum_{\substack{w\in \AAA_t \\ w\subset D}} e^{\Phi(w)}
+ \mu(D^c) \log \sum_{\substack{w\in \AAA_t \\ w\subset D^c}} e^{\Phi(w)}
+ H(\mu(D))
\]
where $D^c = X\setminus D$, and $\Phi$ is as in \eqref{eqn:Phiw}.
\end{lemma}
\begin{proof}
Abramov's formula \cite{lA59} gives $h_\mu(f_t) = t h_\mu(f_1)$ for all $t\in \RR^+$, and Lemma \ref{lem:genAt} gives $h_\mu(f_t,\AAA_t) = h_\mu(f_t)$, so
\[
tP_\mu(\phi)
= h_\mu(f_t,\AAA_t) + \int \Phi_0(x,t)\,d\mu 
\leq \sum_{w\in \AAA_t} \mu(w) \big(-\log \mu(w) + \Phi(w) \big).
\]
Let $\WWW = \{ w \in \AAA_t \mid w\subset D\}$, and write $\WWW^c = \AAA_t \setminus \WWW$.  Breaking up the above sum and normalizing, we have
\begin{align*}
tP_\mu(\phi) &\leq \sum_{w\in \WWW} \mu(w) \big(\Phi(w) - \log\mu(w)\big) +
\sum_{w\in \WWW^c} \mu(w) \big(\Phi(w) - \log\mu(w)\big) \\
&= \mu(D) \sum_{w\in \WWW} \frac{\mu(w)}{\mu(D)} \left( \Phi(w) - \log \frac{\mu(w)}{\mu(D)} \right) \\
&\qquad
 + \mu(D^c) \sum_{w\in \WWW^c} \frac{\mu(w)}{\mu(D^c)} \left( \Phi(w) - \log \frac{\mu(w)}{\mu(D^c)} \right) \\
 &\qquad
 + (-\mu(D)\log\mu(D) - \mu(D^c)\log\mu(D^c)).
\end{align*}
Recall that for non-negative $p_i$ with $\sum p_i=1$ and arbitrary $a_i\in \RR$ we have $\sum_i p_i(a_i - \log p_i) \leq \log \sum_i e^{a_i}$; the conclusion of Lemma \ref{lem:AD} follows by applying this to the first sum with $p_w = \mu(w)/\mu(D)$, $a_w = \Phi(w)$, and the second sum with $p_w = \mu(w)/\mu(D^c)$, $a_w = \Phi(w)$.
\end{proof}

Now we return to the proof of Proposition \ref{prop:all-pressure}.  Let $\eps>0$ be as in the hypothesis, and let $\mu$ be ergodic with $P_\mu(\phi) > \Pexp(\phi,\eps)$, so that $\mu$ is almost expansive at scale $\eps$.  Fix $\alpha>0$.  Given $s\in \RR^+$, consider the set
\[
X_s := \{x\in X \mid \Gamma_\eps(x) \subset f_{[-s,s]}(x)\}.
\]
We have $\bigcup_s X_s = X\setminus \NE(\eps)$, so there is $s$ such that $\mu(X_s) > 1-\alpha$. 

Now, we fix $t =s/\alpha$, and for an arbitrary $r>0$, we write
\[
 B_{[-r, t+r]} (x, \eps) :=  \{ y : d(f_\tau x, f_\tau y) \leq \eps \text{ for } \tau \in [-r, t+r]\} 
\] 
For any $x \in X\setminus \NE(\eps)$, we have
\[
 \Gamma_\eps (x) = \bigcap_{r>0} B_{[-r, t+r]} (x, \eps).
\]
In particular, given $s$ as above and $x\in X_s$, we see that $\bigcup_{y \in f_{[-s,s]}(x)}B_t(y, \alpha)$ is an open set which contains $\Gamma_\eps (x)$, so there is $r=r(x)$ so that  $B_{[-r, t+r]} (x, \eps) \subset \bigcup_{y \in f_{[-s,s]}(x)}B_t(y, \alpha)$. Now, for $r>0$, let 
\[
Y_r := \Bigg\{x \in X_s : B_{[-r, t+r]} (x, \eps) \subset \bigcup_{y \in f_{[-s,s]}(x)}B_t(y, \alpha)\Bigg\}.
\]
We have $\bigcup_{r>0} Y_r = X_s$, so we can fix $r$ sufficiently large so that $\mu(Y_r) > 1-\alpha$. We now pass to the set of points whose orbits spend a large proportion of time in $Y_r$.
Given $n\in \NN$, consider the set
\[
Z_n = \{x\in X \mid \Leb\{\tau\in [0,nt] \mid f_\tau(x) \in Y_r \} > (1-\alpha)nt\},
\]
and note that $\lim_{n\to\infty} \mu(Z_n) = 1$ by the Birkhoff ergodic theorem.  Take $N$ large enough that $\mu(Z_n) > 1-\alpha$ for all $n\geq N$.  The following lemma gives us a regularity property for the potential $\phi$ for points in $Z_n$.

\begin{lemma}\label{lem:ZnPhi}
Given $x\in Z_n$ and $y\in B_{nt}(x,\eps)$, we have
\begin{equation}\label{eqn:ZnPhi}
|\Phi_0(x,nt) - \Phi_0(y,nt)| \leq (8\alpha nt + 4r)\|\phi\| + nt \Var(\phi,\alpha).
\end{equation}
\end{lemma}
\begin{proof}
Let $\mathbf{T} = \{\tau\in [0,nt] \mid f_\tau(x) \in Y_r\}$ and choose $\xi>0$ such that $\Leb ([0,nt] \setminus \mathbf{T}) + \xi n < \alpha n t$ (here we use that $x\in Z_n$).  Define $\tau_1,\dots, \tau_k \in [r,nt-r]$ iteratively as follows: let $\tau_1' = \inf (\mathbf{T} \cap [r,nt])$, and then given $\tau_i'$, choose any  $\tau_i \in \mathbf{T} \cap [\tau_i', \tau_i' + \xi]$, and put $\tau_{i+1}' = \inf(\mathbf{T} \cap [\tau_i + t + 2s,nt])$.

\begin{figure}[htbp]
\includegraphics[width=.65\textwidth]{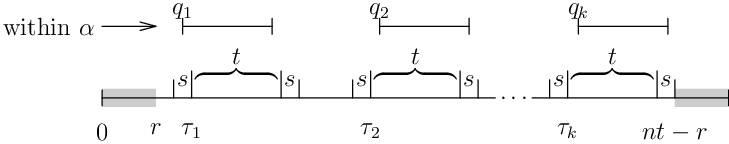}
\caption{Proving Lemma \ref{lem:ZnPhi}.}
\label{fig:tau-i}
\end{figure}

It follows from the definition of $\tau_i$ and properties of $\mathbf{T}$ that
\begin{itemize}
\item $\tau_{i+1} > \tau_i + t+2s$ for every $i$;
\item  $\sum_{i=1}^{k-1} (\tau_{i+1} - (\tau_i + t+2s)) \leq \alpha nt$;  and
\item $f_{\tau_i}(x) \in Y_r$ for every $i$.
\end{itemize}
Since $ f_{\tau_i}y\in B_{[-r, t+r]} (f_{\tau_i}x, \eps)$, the third property gives $f_{\tau_i}y\in B_t(f_{q_i}x,\alpha)$ for some $q_i \in [\tau_i-s,\tau_i+s]$; see Figure \ref{fig:tau-i}.  Thus
\[
|\Phi_0(f_{\tau_i}y,t) - \Phi_0(f_{q_i}x,t)| \leq 2s\|\phi\| + t\Var(\phi,\alpha).
\]
The first two properties give
\begin{align*}
\left|\Phi_0(y,nt) - \sum_{i=1}^k \Phi_0(f_{\tau_i}y, t)\right| &\leq (\alpha nt + 2r + 2sn)\|\phi\|, \\
\left|\Phi_0(x,nt) - \sum_{i=1}^k \Phi_0(f_{q_i}x, t)\right| &\leq (\alpha nt + 2r + 2sn)\|\phi\|,
\end{align*}
and putting it all together we have
\begin{align*}
|\Phi_0(x,nt) - \Phi_0(y,nt)| &\leq 2(\alpha nt + 2r + 2sn) \|\phi\| + 2sn\|\phi\| + nt\Var(\phi,\alpha) \\
&\leq (8\alpha nt + 4r)\|\phi\| + nt \Var(\phi,\alpha),
\end{align*}
which proves the lemma.
\end{proof}

To finish the proof of Proposition \ref{prop:all-pressure}, for $n\geq N$,  let $E_n$ be any maximizing $(nt,\eps/2)$-separated set for $\Lambda(X,\eps,nt)$, and let $\AAA_n$ be an adapted partition for $E_n$.  Let $D_n = \bigcup \{ w \mid w\in \AAA_n, w \cap Z_n \neq \emptyset\}$ and note that $\mu(D_n) > 1-\alpha$.  With $\alpha<1/2$, Lemma \ref{lem:AD} gives
\[
ntP_\mu(\phi) \leq \log \sum_{\substack{w\in \AAA_t \\ w\subset D_n}} e^{\Phi(w)}
+ \alpha \log \sum_{\substack{w\in \AAA_t \\ w\subset D_n^c}} e^{\Phi(w)}
+ H(\alpha).
\]
We need to get estimates on the sums.  Note that there is a 1-1 correspondence between elements of $E_n$ and elements of $\AAA_n$, and that given $x\in E_n$ and $w\in \AAA_n$ with $x\in w$, we have $\Phi(w) \leq \Phi_{\eps}(x,nt)$.  Thus
\begin{equation}\label{eqn:tPmu}
ntP_\mu(\phi) \leq \log \sum_{x\in E_n \cap D_n} e^{\Phi_\eps(x,nt)}
+ \alpha \log \sum_{x\in E_n \cap D_n^c} e^{\Phi_\eps(x,nt)} + H(\alpha).
\end{equation}
To control the first sum, we use Lemma \ref{lem:ZnPhi} and get
\[
\log \sum_{x\in E_n \cap D_n} e^{\Phi_\eps(x,nt)}
\leq (8\alpha \|\phi\| +\Var(\phi,\alpha))nt+ 4T\|\phi\| + \log \Lambda(X,\eps,nt),
\]
while for the second we can put $Q>P(\phi,\eps,2\eps)$ and obtain a constant $C$ such that for every $n$ we have
\[
\sum_{x\in E_n \cap D_n^c} e^{\Phi_\eps(x,nt)}
\leq \Lambda(X,\eps,2\eps,nt) \leq C e^{Q nt}.
\]
Dividing both sides of \eqref{eqn:tPmu} by $nt$ and sending $n\to\infty$, these bounds give
\[
P_\mu(\phi) \leq 8\alpha\|\phi\| + \Var(\phi,\alpha) + P(\phi,\eps) + \alpha Q.
\]
Since $\alpha>0$ was arbitrary, we conclude that $P(\phi,\eps) \geq P_\mu(\phi)$.  Taking a supremum over all ergodic $\mu$ with $P_\mu(\phi) > \Pexp(\phi,\eps)$ gives $P(\phi,\eps) \geq P(\phi)$, which completes the proof of Proposition \ref{prop:all-pressure}.
\end{proof}

\subsection{Approximating invariant sets with adapted partitions} \label{sec:approxinv} 

In addition to the generating properties from the previous section, adapted partitions are useful for approximating invariant sets with Bowen balls; this will be used in the proof of uniqueness in \S\S\ref{sec:unique}--\ref{sec:ergodic}.  The following proposition is inspired by an approximation lemma of Bowen \cite[Lemma 2]{rB73}, 
which plays a key role in Franco's proof.

\begin{proposition}\label{prop:approximating}
Let $F$ be a continuous flow on a compact metric space $X$, and suppose $\mu\in \MMM_F(X)$  is almost expansive at scale $\eps$.  Let $\gamma \in (0,\eps/2]$, and for each $t>0$, let  $\AAA_t$ be an adapted partition for a $(t,\gamma)$-separated set of maximal cardinality. 
Let $Q \subset X$ be a measurable $F$-invariant set.  Then for every $\alpha>0$, there exists $t_0$ so that if $t\geq t_0$, we can find $U\subset \AAA_t$ such that $\mu(U\bigtriangleup Q) < \alpha$.
\end{proposition}
\begin{proof}
If $\mu(Q)=0$ we take $U=\emptyset$, so from now on we assume $\mu(Q)>0$.
For a set $w \subset X$ and $s\in \RR^+$, let
\begin{equation}
\diam_{[-s, s]}w = \sup_{x_1,x_2 \in w} \inf_{t_1,t_2 \in [-s,s]}d(f_{t_1} x_1, f_{t_2} x_2),
\end{equation}
and for a partition $\AAA$, let $\diam_{[-s,s]} \AAA = \max \{ \diam_{[-s,s]} w : w \in \AAA\}$.   
Bowen proved that the conclusion of Proposition \ref{prop:approximating} holds if the partitions $\AAA_t$ satisfy $\diam_{[-s,s]} \AAA_t \to 0$ for some $s\in \RR^+$.
In our setting we do not have uniform convergence and so we first need to pass to a set of large measure. For $s \in \RR^+$,  
let $X_s = \{x \mid \Gamma_\eps(x) \subset f_{[-s,s]}(x)\}$.   Fix $\beta>0$. Since $\mu$ is almost expansive at scale $\eps$, the arguments 
which precede Lemma \ref{lem:ZnPhi} show that there exists $s$ such that $\mu(X_s) > 1-\beta$ and for every $x\in X_s$, $\diam_{[-s,s]} B_{[-t, t]} (x, \eps) \to 0$ as $t\to \infty$.  

Let $\AAA'_t:= f_{t/2}\AAA_t$, and write $w_t(x)$ for the element of the partition $\AAA'_t$ which contains $x$, and observe that by the construction of $\AAA'_t$, for every $x\in X$ there exists a point $x'$ so that $w_t(x) \subset B_{[-t/2, t/2]}(x', \gamma)$.  It follows that  $w_t(x) \subset B_{[-t/2, t/2]}(x, 2 \gamma)$, and thus $\diam_{[-s,s]} w_t (x) \to 0$ for $\mu$-a.e $x \in X_s$, and by Egorov's theorem there is $X'_s\subset X_s$ with $\mu(X_s\setminus X'_s) < \beta$ such that the convergence is uniform on $X'_s$. 

By restricting our attention to $Q' := X'_s \cap Q$ we can follow Bowen's argument from \cite[Lemma 2]{rB73}.

Let $K_1 \subset Q'$ and $K_2 \subset X \setminus Q$ be compact sets with $\mu(Q' \setminus K_1) < \beta$ and $\mu(X\setminus (Q\cup K_2)) < \beta$.  For $i=1,2$ let $K_i^s = \{f_t(x) \mid x\in K_i, t\in [-s,s]\}$ and note that $K_1^s,K_2^s$ are compact; moreover, they are disjoint, since $K_1^s \subset Q$ and $K_2^s\subset X\setminus Q$ by $F$-invariance of $Q$.  Thus $K_1^s$ and $K_2^s$  are uniformly separated by some distance $\gamma>0$; that is,
\begin{equation}
d(f_{t_1}(x_1), f_{t_2}(x_2)) \geq \gamma \text{ for all } x_i\in K_i \text{ and } t_i\in [-s,s].
\end{equation}
This gives $d_{[-s,s]}(K_1,K_2) > \gamma$.  By the uniform convergence obtained above on $Q'$, there is $t_0\in \RR^+$ such that $\diam_{[-s,s]} w_t(x) < \gamma$ for every $t \geq t_0$ and $x\in Q'$. It follows that if $t \geq t_0$, then every $w\in \AAA'_t$ with $w\cap K_1 \neq \emptyset$ has $w\cap K_2 = \emptyset$.  Let $U' = \bigcup \{ w\in \AAA'_t \mid w\cap K_1 \neq \emptyset\}$.
Then $K_1 \subset U'$ and $K_2 \cap U' =\emptyset$.
Hence
\begin{align*}
\mu (U' \bigtriangleup Q) & = \mu(U'\setminus Q) + \mu(Q\setminus U') 
\leq \mu(X \setminus (Q \cup K_2)) + \mu(Q\setminus K_1)\\
&\leq \beta + \mu(Q\setminus Q') + \mu(Q' \setminus K_1) 
\leq \beta + 2\beta + \beta = 4\beta,
\end{align*}
and we were free to choose $\beta < \alpha/4$. To complete the proof of Proposition \ref{prop:approximating}, let $U \subset \AAA_t$ be defined to be $U = f_{-t/2}U'$. By $F$-invariance of the measure $\mu$ and the set $Q$, we have
\[
\mu (U \bigtriangleup Q) = \mu (f_{-t/2}(U \bigtriangleup Q)) = \mu(U'\bigtriangleup Q) <\alpha. \qedhere
\]
\end{proof}


\section{Proof of Theorem \ref{thm:flowsD}}\label{sec:flows-pf}
We adapt the strategy of the proof in \cite{CT3}, which in turn was inspired by the proof of Bowen \cite{rB74}.
We use $\delta$ to represent the scale at which specification occurs, and $\eps$ for the scale at which expansivity and the Bowen property occur.  We use $\gamma,\theta$, etc.\ to represent other scales, which usually lie in between $\delta$ and $\eps$.  Many of the intermediate lemmas in this section hold without any specification or expansivity assumption, so we take care to state exactly which assumptions are required. 
\subsection{Lower bounds on $X$}\label{sec:lower-X}
\begin{lemma}\label{lem:leqproduct}
For every $\gamma>0$ and $t_1,\dots,t_k>0$ we have
\begin{equation}\label{eqn:leqproduct}
\Lambda(X,2\gamma, t_1+\cdots+t_k) \leq \prod_{j=1}^k \Lambda(X,\gamma,\gamma,t_j).
\end{equation}
\end{lemma}
\begin{proof}
Write $T=t_1 + \cdots + t_k$.  We will also need to consider the partial sums $r_j = t_1 + \cdots + t_{j-1}$ for each $2\leq j\leq k$, where we put $r_1=0$.  Let $E$ be a maximizing $(T,2\gamma)$-separated set, and similarly for each $j$ let $E_j$ be a maximizing $(t_j,\gamma)$-separated set.  Each $E_j$ is $(t_j,\gamma)$-spanning, and so we can define a map $\pi\colon E\to E_1\times \cdots \times E_k$ by $\pi(x) = (x_1,\dots, x_k)$, where $x_j\in E_j$ is such that $f_{r_j}(x) \in \overline B_{t_j}(x_j,\gamma)$.  

We claim that $\pi$ is injective.  To see this, observe that for all distinct $x,y\in E$ there exists $j$ such that $d_{t_j}(f_{r_j}x,f_{r_j}y) > 2\gamma$, and thus
\begin{align*}
d_{t_j}(x_j, y_j) &\geq
d_{t_j}(f_{r_j}x, f_{r_j}y) - d_{t_j}(f_{r_j}x, x_j) - d_{t_j}(f_{r_j}y, y_j) \\
&> 2\gamma - \gamma - \gamma = 0,
\end{align*}
so $x_j \neq y_j$ and thus $\pi(x) \neq \pi(y)$.  It follows that
\begin{equation}\label{eqn:pi-1}
\Lambda(X, 2\gamma,T) = \sum_{x\in E} e^{\Phi_{0}(x,T)} 
= \sum_{(x_1,\dots,x_k)\in \pi(E)} e^{\Phi_{0}(\pi^{-1}(x_1,\dots,x_k),T)}.
\end{equation}
Given $x\in E$ with $\pi(x) = (x_1,\dots, x_k)$, we observe that $f_{r_j}(x) \in \overline B_{t_j}(x_j,\gamma)$. 
Thus
$
\Phi_{0}(f_{r_j}x,t_j) \leq \Phi_{\gamma}(x_j,t_j),
$
and so
\[
\Phi_{0}(x,T) \leq \sum_{j=1}^k \Phi_{0}(f_{r_j}x,t_j) \leq \sum_{j=1}^k \Phi_{\gamma}(x_j,t_j).
\]
Together with \eqref{eqn:pi-1}, this gives
\begin{align*}
\Lambda(X,2\gamma,T) &\leq \sum_{(x_1,\dots,x_k)\in \pi(E)} e^{\sum_{j=1}^k \Phi_{\gamma}(x_j,t_j)} \\
&\leq \sum_{x_1\in E_1} \cdots \sum_{x_k\in E_k} \prod_{j=1}^k e^{\Phi_{\gamma}(x_j,t_j)}
= \prod_{j=1}^k \Lambda(X,\gamma,\gamma,t_j),
\end{align*}
completing the proof of Lemma \ref{lem:leqproduct}.
\end{proof}

\begin{lemma}\label{lem:X-lower}
Let $\eps$ satisfy $\Pexp(\phi,\eps) < P(\phi)$. For every $t\in \RR^+$ and $0<\gamma\leq \eps/4$, we have the inequality
$\Lambda(X,\gamma,\gamma,t) \geq e^{tP(\phi)}$.
\end{lemma}
\begin{proof}
For every $k\in \NN$, Lemma \ref{lem:leqproduct} gives $\Lambda(X,2\gamma,kt) \leq (\Lambda(X,\gamma,\gamma,t))^k$.  It follows that $\frac 1{kt} \log \Lambda(X,2\gamma,kt) \leq \frac{1}{t}\log \Lambda(X,\gamma, \gamma, t)$. The left hand side converges to $P(\phi, 2\gamma)$ as $k\to\infty$. The result follows by Proposition \ref{prop:all-pressure}.
\end{proof}

This lemma demonstrates why we need to consider `two-scale' pressure. If $\phi$ had the Bowen property globally (at scale $\gamma$), then by Remark \ref{rmk:two-scale-b}, we would have
\[
\Lambda(X, \gamma, t) \geq e^{-K} \Lambda(X,\gamma,\gamma,t) \geq e^{-K}e^{tP(\phi)},
\]
so we can reduce to the standard partition sum. However, if $\ph$ is not Bowen globally, then we only obtain
\[
\Lambda(X, \gamma, t) \geq e^{-t\Var(\phi, \gamma)} \Lambda(X,\gamma,\gamma,t)\geq e^{-t\Var(\phi, \gamma)}e^{tP(\phi)}.
\]
which is why we must work with the `two-scale' partition sums. In our setting, to obtain lower bounds on the standard partition sums $\Lambda(X, \gamma, t)$, we require a more in-depth analysis, and all the hypotheses of our theorem. This is carried out in Proposition \ref{Prop:lowerG}.

\subsection{Upper bounds on $\GGG$}

Now suppose that $\GGG\subset X\times \RR^+$ has specification at scale $\delta>0$ for $t \geq T_0$.

\begin{proposition}\label{prop:gluing}
Suppose that $\GGG$ has specification at scale $\delta>0$ for $t \geq T_0$ with maximum gap size $\tau$, and $\zeta>\delta$ is such that $\phi$ has the Bowen property on $\GGG$ at scale $\zeta$.  
Then, for every $\gamma > 2\zeta$,  there is a constant $C_1>0$ so that for every $k\in \NN$ and $t_1,\dots,t_k \geq T_0$, writing  $T := \sum_{i=1}^{k} t_i + (k-1)\tau$, and $\theta := \zeta-\delta$, we have
\begin{equation}\label{eqn:gluing}
\prod_{j=1}^k \Lambda(\GGG,\gamma,t_j)
\leq C_1^k \Lambda(X,\theta,T).
\end{equation}
\end{proposition}
\begin{proof}
Let $\lambda_j < \Lambda(\GGG,\gamma,t_j)$ be arbitrary, and let $E_1,\dots,E_k \subset \GGG$ be $(t_j,\gamma)$-separated sets such that  
$
\sum_{x\in E_j} e^{\Phi_{0}(x,t)} > \lambda_j.
$

By specification, for every $\xx = (x_1,\dots,x_k) \in \prod_{j=1}^k E_j$ there are $y=y(\xx)\in X$ and $\ttau=\ttau(\xx) = (\tau_1,\dots,\tau_{k-1}) \in [0,\tau]^{k-1}$ such that $y=y(\xx)\in B_{t_1}(x_1,\delta)$, $f_{t_1 + \tau_1}(y)\in B_{t_2}(x_2,\delta)$, and so on.  Moreover, $\tau_j$ depends only on $x_1,\dots,x_{j+1}$.

Let $E\subset X$ be maximizing $(\theta,T)$-separated, and hence $(\theta,T)$-spanning.  Let $p\colon X\to E$ be the map that takes $x\in X$ to the point in $E$ that is closest to it in the $d_T$-metric.  Let $\pi = p \circ y \colon \prod E_j \to E$, as shown in the following commutative diagram.

\begin{equation}\label{eqn:pihatpi}
\xymatrix{
E_1\times\cdots\times E_k \ar[d]_{(y,\ttau)} \ar[rd]^{\pi} \\
X\times [0,\tau]^{k-1} \ar[r]_-{p} & E 
}
\end{equation}
Writing $s_j (\xx)= \sum_{i=1}^{j} t_i + \sum_{i=1}^{j-1}\tau_i$, and $s_0(\xx)=\tau_0(\xx)=0$, we observe that for every $\xx\in \prod E_j$ and every $1\leq j\leq k$, we have
\begin{equation}\label{eqn:distances}
d_{t_j}(x_j,f_{s_{j-1}+\tau_{j-1}}(\pi\xx)) \leq \delta + \theta = \zeta. 
\end{equation}
We will use the map $\pi$ to compare the two sides of \eqref{eqn:gluing}.  For this we will need the following two lemmas controlling the multiplicity of $\pi$, and the difference between the weights of the orbit segments.

\begin{lemma}\label{lem:multiplicity}
There is a constant $C_2\in \RR$ such that $\#\pi^{-1}(z) \leq C_2^k$  for every $z\in E$.
\end{lemma}

\begin{lemma}\label{lem:weights}
There is a constant $C_3\in \RR$ 
such that every $\xx\in \prod E_j$ has $\Phi_0(\pi\xx,T) \geq -kC_3 + \sum_j \Phi_0(x_j,t_j)$.
\end{lemma}

We remark that in addition to the dynamics $(X,F)$, the constant $C_2$ depends on $\tau$ and $\gamma$, while $C_3$ depends on $\tau$, $\|\phi\|$, and the constant $K$ from the Bowen property.  Before proving Lemmas \ref{lem:multiplicity} and \ref{lem:weights}  we show how they complete the proof of Proposition \ref{prop:gluing}.  We see that
\begin{equation}\label{eqn:Lambda-X-theta-T}
\begin{aligned}
\Lambda(X,\theta,0,T) &= \sum_{z\in E} e^{\Phi_{0}(z,T)} 
\geq C_2^{-k} \sum_{\xx\in \prod_j E_j} e^{\Phi_{0}(\pi\xx,T)} \\
&\geq C_2^{-k} e^{-kC_3} \sum_{\xx\in \prod_j E_j} \prod_j e^{\Phi_0(x_j,t_j)}
\geq (C_2e^{C_3})^{-k} \prod_j \lambda_j, 
\end{aligned}
\end{equation}
where the first inequality uses Lemma \ref{lem:multiplicity} and the second uses Lemma \ref{lem:weights}.  Since $\lambda_j$ can be taken arbitrarily close to $\Lambda(\GGG,\gamma,t_j)$, this establishes \eqref{eqn:gluing}, and so it only remains to prove the lemmas.

\begin{proof}[Proof of Lemma \ref{lem:multiplicity}] 
Let $\gamma' = \gamma - 2\zeta>0$, and let $m\in \NN$ be large enough so that writing $\zeta':=\tau/m$,  we have $d(x,f_sx)<\gamma'$ for every $x\in X$ and $s\in (-\zeta',\zeta')$. We partition the interval $[0,k \tau]$ into $km$ sub-intervals $I_1,\dots, I_{km}$ of length $\zeta'$, denoting this partition $P$.  
Given $\xx\in \prod E_j$, take the sequence $n_1, \ldots, n_k$ so that
\[
\tau_1(\xx) + \cdots + \tau_i(\xx) \in I_{n_i} \text{ for every } 1\leq i \leq k-1.
\]
Now let $\ell_1=n_1$ and $\ell_{i+1}=n_{i+1}-n_i$ for  $1\leq i \leq k-2$, and let $\ell(\xx):= (\ell_1,\dots,\ell_{k-1})$. Since $\tau_{i+1}(\xx)\in[0, \tau]$, we have $n_i \leq n_{i+1} \leq n_i+m$ for each $i$, and thus $\ell(\xx) \in  \{0,\dots, m-1\}^{k-1}$. 

Now given $\bar\lll 
\in \{0,\dots,m-1\}^{k-1}$, let $E^{\bar\lll} \subset \prod E_j$ be the set of all $\xx$ such that $\lll(\xx) = \bar\lll$. Note that if $\xx, \xx' \in E^{\bar\lll}$ and $i \in \{1, \ldots, k-1\}$, then by construction, $\tau_1(\xx) + \cdots + \tau_i(\xx)$ and $\tau'_1(\xx) + \cdots + \tau'_i(\xx)$ belong to the same element of the partition $P$.

We show that $\pi$ is 1-1 on each $E^{\bar\lll}$. Fix $\bar\lll$ and let $\xx,\xx'\in E^{\bar\lll}$ be distinct.  Let $j$ be the smallest index such that $x_j\neq x_j'$. Write $\tau_i=\tau_i(\xx)$ and $\tau_i' = \tau_i(\xx')$.  Let $r = \sum_{i=1}^j (t_i + \tau_i)$ and $r' = \sum_{i=1}^j (t_i + \tau_i')$.  Since $\sum_{i=1}^j \tau_i$ and $ \sum_{i=1}^j \tau_i$ belong to the same element of $P$, then $|r-r'| = |\sum_{i=1}^j \tau_i - \sum_{i=1}^j \tau_i'|  < \zeta'$.

Because $x_j\neq x_j'\in E_j$ and $E_j$ is $(t_j,\gamma)$-separated, we have $d_{t_j}(x_j,x_j') >  \gamma$.  Now we have
\[
d_T(\pi\xx, \pi\xx') \geq d_{t_j}(f_{r}\pi\xx,f_{r}\pi\xx')  \geq d_{t_j}(f_{r}\pi\xx, f_{r'}\pi\xx') - \gamma',
\]
where the $\gamma'$ term comes from the fact that $d_{t_j}(f_{r}\pi\xx',f_{r'}\pi\xx') \leq \gamma'$ by our choice of $\zeta'$.  For the first term, observe that
\[
d_{t_j}(f_{r}\pi\xx,f_{r'}\pi\xx') \geq d_{t_j}(x_j, x_j') - d_{t_j}(x_j, f_{r}\pi\xx) - d_{t_j}(f_{r'}\pi\xx', x_j').
\]
Using \eqref{eqn:distances} gives
\[
d_T(\pi\xx,\pi\xx') > \gamma - 2\zeta - \gamma' = 0,
\]
and so $\pi\xx\neq\pi\xx'$, and thus $\pi$ is 1-1 on each $E^{\bar\lll}$. Since there are $m^{k-1}$ choices of $\bar\lll$, this completes the proof of Lemma \ref{lem:multiplicity}.
\end{proof}

\begin{proof}[Proof of Lemma \ref{lem:weights}]
Let $\xx\in \prod_j E_j$.  Noting that $T - \sum_{j=1}^k t_j = (k-1)\tau$, we have 
\[
\Phi_0(\pi\xx,T)
= \int_0^T \phi(f_s(\pi\xx))\,ds 
\geq -(k-1)\tau\|\phi\| + \sum_{j=1}^{k}  \Phi_0(f_{s_{j-1}+\tau_{j-1}}(\pi\xx),t_j).
\]
Moreover, since from \eqref{eqn:distances} we have $f_{s_{j-1}+\tau_{j-1}}(\pi\xx)\in \overline B_{\zeta}(x_j,t_j)$ for every $j$, we can use the Bowen property at scale $\zeta$ to get
\[
\Phi_0(f_{s_{j-1}+\tau_{j-1}}(\pi\xx),t_j) \geq \Phi_0(x_j,t_j) - K.
\]
We conclude that
\[
\Phi_0(\pi\xx,T) \geq -k(\tau\|\phi\| + K) + \sum_{j=1}^k \Phi_0(x_j,t_j),
\]
which completes the proof of Lemma \ref{lem:weights}.
\end{proof}

As explained above, this completes the proof of Proposition  \ref{prop:gluing}, by the computation in \eqref{eqn:Lambda-X-theta-T}.
\end{proof}

The following corollary extends the statement of Proposition \ref{prop:gluing} to the more general partition sums with $\eta>0$.
\begin{corollary}\label{rmk:gluing}
Let $\GGG, \delta, \zeta, \tau, \gamma, \phi, t_j, T, \theta, C_1$ be as in Proposition \ref{prop:gluing}, and let $\eta_1, \eta_2 \geq 0$. Suppose further that the Bowen property for $\phi$ holds at scale $\eta_1$. Then
\begin{equation}\label{eqn:gluing2}
\prod_{j=1}^k \Lambda(\GGG,\gamma,\eta_1,t_j)
\leq (e^{K}C_1)^k \Lambda(X,\theta,\eta_2,T),
\end{equation}
where $K$ is the distortion constant from the Bowen property.
\end{corollary}
\begin{proof}
Observe that replacing the term $\Lambda(X,\theta,0,T)$ with $\Lambda(X,\theta,\eta_2,T)$ increases the right-hand side of \eqref{eqn:gluing}, so it suffices to prove \eqref{eqn:gluing2} when $\eta_2=0$. If $\phi$ has the Bowen property on $\GGG$ as scale $\eta_1$, then it is easy to show that
\[
\Lambda(\GGG,\gamma,\eta_1,t) \leq e^K \Lambda(\GGG,\gamma,0,t)
\]
for every $\gamma,t>0$.  Thus, \eqref{eqn:gluing}  yields the required inequality.
\end{proof}

The key consequence of Proposition \ref{prop:gluing} is the following upper bound on partition sums over $\GGG$.

\begin{proposition}\label{prop:upper-G}
Suppose that $\GGG\subset X\times \RR^+$ has tail specification at scale $\delta>0$. Suppose that $\gamma>2\delta$ and $\eta \geq 0$. Suppose that $\phi$ has the Bowen property on $\GGG$ at scale $\max\{\gamma/2, \eta\}$.  
Then there is $C_4$ such that for every $t \geq0$ we have
\begin{equation}\label{eqn:upper-G}
\Lambda(\GGG,\gamma,\eta,t) \leq C_4 e^{tP(\phi)}.
\end{equation}
\end{proposition}
\begin{proof}
Let $\zeta>0$ satisfy $2\delta<2\zeta <\gamma$, and thus by assumption, $\GGG$ has the Bowen property at scale $\zeta$. Let $\theta :=\zeta-\delta$. There is a $T_0 \in \RR^+$ so that by Corollary \ref{rmk:gluing}, we have
\[
(\Lambda(\GGG,\gamma,\eta,t))^k \leq (e^KC_1)^k \Lambda(X,\theta,k(t+\tau))
\] 
for every $k\in \NN$ and $t\geq T_0$.
Taking logarithms and dividing by $k$ yields
\[
\log \Lambda(\GGG,\gamma,\eta,t) \leq \log(e^KC_1) + (t+\tau)\frac 1{k(t+\tau)} \log \Lambda(X,\theta,k(t+\tau)).
\]
Sending $k\to\infty$ we get
\[
\log \Lambda(\GGG,\gamma,\eta,t) \leq K+\log C_1 + (\tau+t) P(\phi, \theta);
\]
since $P(\phi,\theta) \leq P(\phi)$ this gives
\[
\Lambda(\GGG,\gamma,\eta,t) e^{-tP(\phi)} \leq C_1 e^{K+\tau P(\phi)}
\]
for all $t\geq T_0$, which proves Proposition \ref{prop:upper-G} with
\[
C_4 = \max\bigg( C_1 e^{K+\tau P(\phi)},\ \sup_{t\in [0,T_0]} \Lambda(\GGG,\gamma,\eta,t) e^{-tP(\phi)} \bigg).
\qedhere
\]

\end{proof}

\subsection{Lower bounds on $\GGG$}

\begin{lemma}\label{lem:many-in-G}
Fix a scale $\gamma > 2\delta>0$, 
and let $(\PPP,\GGG,\SSS)$ be a decomposition for $\DDD \subseteq X \times \RR^+$ such that
\begin{enumerate}
\item $\GGG$ has tail specification at scale $\delta$;
\item $\phi$ has the Bowen property on $\GGG$ at scale $3 \gamma$; and
\item $P(\DDD^c,\phi,2\gamma, 2\gamma) < P(\phi) \text{ and }P(\altP \cup \altS,\gamma , 3 \gamma) < P(\phi)$,
\end{enumerate}
Then for every $\alpha_1,\alpha_2>0$, there exists $M\in \RR^+$ and $T_1 \in \RR^+$ such that the following is true:
\begin{itemize}
\item for any $t\geq T_1$ and $\CCC\subset X\times \RR^+$ such that $\Lambda(\CCC,2\gamma,2\gamma, t)\geq \alpha_1 e^{tP(\phi)}$, we have $\Lambda(\CCC\cap \GGG^M,2\gamma,2\gamma,t) \geq (1-\alpha_2) \Lambda(\CCC,2\gamma,2\gamma,t)$. 
\end{itemize}
\end{lemma}
\begin{proof}
Fix $\beta_1>0$ such that 
\[
P(\DDD^c,\phi,2\gamma, 2 \gamma) < P(\phi)-2\beta_1 \text{ and }P(\altP \cup \altS,\gamma , 3 \gamma) < P(\phi) -2\beta_1,
\]
 and so there is  $C_5\in \RR^+$ such that
\begin{equation}\label{eqn:C5}
\begin{aligned}
P(\altP \cup \altS,\gamma , 3 \gamma) &\leq C_5 e^{t(P(\phi) - \beta_1)}, \\
\Lambda(\DDD^c,2\gamma, 2 \gamma, t) &\leq C_5 e^{t(P(\phi) - \beta_1)}
\end{aligned}
\end{equation}
for all $t>0$. We consider $t$ sufficiently large
so that 
\begin{equation} \label{eqn:tlarge}
C_5 e^{t(P(\phi) - \beta_1)}\leq \tfrac{1}{2}\alpha_1\alpha_2 e^{tP(\phi)}.
\end{equation}
For an arbitrary $\beta_2>0$, let $E_t \subset \CCC_t$ be a $(t, 2\gamma)$-separated set so that
\begin{equation}\label{eqn:Et}
\sum_{x\in E_t} e^{\Phi_{2\gamma}(x,t)} > \Lambda(\CCC,2\gamma, 2 \gamma, t)-\beta_2.
\end{equation}
Writing $i \vee j $ for $\max\{i,j\}$, we split the sum in \eqref{eqn:Et} into three terms
\begin{equation}\label{eqn:Et2}
\sum_{\substack{x\in E_t \\ (x,t) \in \DDD^c}} e^{\Phi_{2\gamma}(x,t)}
+ \sum_{\substack{x\in E_t \\ (p \vee s)(x,t) \leq M}} e^{\Phi_{2\gamma}(x,t)}
+ \sum_{\substack{x\in E_t \\ (p \vee s)(x,t) > M}} 
e^{\Phi_{2\gamma}(x,t)}
\end{equation}
The first sum is at most $\Lambda(\DDD^c,2\gamma,2 \gamma,t)$, and since $\{x \in E_t \mid (p\vee s)(x,t) \leq M\}$ is a $(t,2\gamma)$-separated subset of $(\CCC \cap \GGG^M)_t$, the second sum
is at most $\Lambda(\CCC\cap \GGG^M,2\gamma,2 \gamma,t)$.

\begin{figure}[htbp]
\includegraphics[width=.6\textwidth]{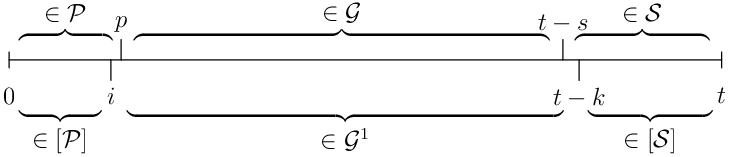}
\caption{Decomposing an orbit segment.}
\label{fig:Cik}
\end{figure}

We work on controlling the final sum in \eqref{eqn:Et2}.
Note that given $x\in E_t$ with $\lfloor p(x,t) \rfloor = i$ and $\lfloor s(x,t) \rfloor = k$, we have
\[
x \in \altP_{i},\quad
f_{i}x \in \GGG^1_{t - (i+k)},\quad
f_{t-k}x \in \altS_{k},
\]
where $\altP,\altS$ are as defined in \eqref{eqn:C1}; see Figure \ref{fig:Cik}.  Given $i,k\in \{1,\dots, \lceil t \rceil\}$, let
\[
C(i,k) = \{x\in E_t \mid (x,t)\in \DDD \text{ and } \lfloor p(x,t) \rfloor = i, \lfloor s(x,t)\rfloor = k \}.
\]
Now for each $i \in \{1, \ldots, \lceil t \rceil \}$, let $E_i^P \subset \altP_i$ be a $(i,\gamma)$-separated set of maximum cardinality; choose $E_j^G \subset \GGG_j^1$ and $E_k^S \subset \altS_k$ similarly.  As in Lemma \ref{lem:leqproduct}, there is an injection $\pi\colon C(i,k) \to E^P_{i} \times E^G_{t-(i+k)} \times E^S_{k}$ given by  $\pi(x) = (x_1, x_2, x_3)$, where 
\begin{itemize}
\item $x_1\in E^P_{i}$ is such that $x\in \overline B_{i}(x_1,\gamma)$;
\item $x_2 \in E^G_{t-(i+k)}$ is such that $f_{i}x \in \overline B_{t-(i+k)}(x_2,\gamma)$;
\item $x_3 \in E^S_{k}$ is such that $f_{t-k}x \in \overline B_{k}(x_3,\gamma)$.
\end{itemize}
Note that for $x \in C(i,k)$ we have
\[
\Phi_{2\gamma}(x,t) \leq \Phi_{3\gamma}(\pi_1 x,i) + \Phi_{3\gamma}(\pi_2 x,t-(i+k)) + \Phi_{3\gamma}(\pi_3 x,k),
\]
so we have the following bound on the final sum in \eqref{eqn:Et2}:
\begin{multline*}
\sum_{\substack{x\in E_t \\ (p \vee s)(x,t) > M}}
e^{\Phi_{2\gamma}(x,t)}
= \sum_{i\vee k > M} \sum_{x\in C(i,k)} e^{\Phi_{2\gamma}(x,t)} \\
\leq 
\sum_{i\vee k > M} \Lamb(\altP,\gamma, 3 \gamma, i) \Lamb(\GGG^1, \gamma, 3 \gamma, t-i-k) \Lamb(\altS,\gamma, 3 \gamma, k).
\end{multline*}
Applying Proposition \ref{prop:upper-G}, we obtain
\begin{align*}
\sum_{\substack{x\in E_t \\ (p \vee s)(x,t) > M}} e^{\Phi_{2\gamma}(x,t)} & \leq  \sum_{i\vee k> M} \Lamb(\altP,\gamma, 3 \gamma, i)  \Lamb(\altS,\gamma, 3 \gamma, k) C_4e^{(t-i-k)P(\phi)} \\
& \leq \sum_{i\vee k> M}C_4 C_5^2e^{(i+k)(P(\phi) - \beta_1)}e^{(t-i-k)P(\phi)} \\
& = C_4 C_5^2e^{tP(\phi)}\sum_{i\vee k> M}e^{-(i+k)\beta_1},
\end{align*}
where the second inequality uses \eqref{eqn:C5}.
Let $M$ be chosen large enough that the sum in the above expression is less than $\frac{1}{2}\alpha_1\alpha_2 C^{-1}_4 C_5^{-2}$.  Then, for $t$ large enough so \eqref{eqn:tlarge} holds, we have
\begin{align*}
\Lambda(\CCC,2\gamma, 2 \gamma, t)-\beta_2 &<  \Lambda(\CCC\cap \GGG^M,2\gamma,2 \gamma,t) + \frac{\alpha_1\alpha_2}{2} e^{tP(\phi)} + \Lambda(\DDD^c,2\gamma,2 \gamma,t)\\
&\leq \Lambda(\CCC\cap \GGG^M,2\gamma,2 \gamma,t) + \frac{\alpha_2}{2} \Lambda(\CCC,2\gamma,2 \gamma,t)+C_5 e^{t(P(\phi)- \beta_1)}\\
&\leq \Lambda(\CCC\cap \GGG^M,2\gamma,2 \gamma,t) + \frac{\alpha_2}{2} \Lambda(\CCC,2\gamma,2 \gamma,t)+ \frac{\alpha_1\alpha_2}{2} e^{tP(\phi)}\\
&\leq \Lambda(\CCC\cap \GGG^M,2\gamma,2 \gamma,t) + \alpha_2 \Lambda(\CCC,2\gamma,2 \gamma,t).
\end{align*}
Since $\beta_2 >0$ was arbitrary and $T_1,M$ were chosen independently of $\beta_2$, this completes the proof of Lemma \ref{lem:many-in-G}.
\end{proof}

\subsection{Consequences of lower bound on $\GGG$}\label{sec:consequences}
 Throughout this section, we assume that $\GGG \subset X \times \RR^+$ and $\delta,\eps>0$ are such that $\GGG$ has tail specification at scale $\delta$ and $\Pexp(\phi,\eps) < P(\phi)$.
The following lemmas are consequences of Lemma \ref{lem:many-in-G}.

\begin{lemma}\label{lem:lowerG} 
Let $\delta,\eps$ be as above.  Fix $\gamma \in (2\delta,\eps/8]$ such that $\phi$ has the Bowen property on $\GGG$ at scale $3\gamma$.  Then for every $\alpha>0$ there are $M\in \NN$ and $T_1 \in \RR$ such that for every $t \geq T_1$ we have
\begin{equation}\label{eqn:lowerG}
\Lambda(\GGG^M,2\gamma,2\gamma,t) \geq (1-\alpha) \Lambda(X,2\gamma,2\gamma,t) \geq (1-\alpha)e^{tP(\phi)}.
\end{equation}
\end{lemma}
\begin{proof}
The inequality $\Lambda(X,2\gamma, 2\gamma,t)\geq e^{tP(\phi)}$ is true for any $0 < 2\gamma\leq \eps/4$ by Lemma \ref{lem:X-lower}.  Since $\gamma>2\delta$ and $\phi$ is Bowen on $\GGG$ at scale $3\gamma$, we can apply Lemma \ref{lem:many-in-G} to $\Lambda(X,2\gamma,2\gamma,t)$ with $\alpha_1=1$ and $\alpha_2=\alpha$, obtaining $M$ and $T_1$ such that \eqref{eqn:lowerG} holds for $t\geq T_1$.
\end{proof}

We now obtain the following lower bound for the standard partition sum  $\Lambda(\GGG^M,2\gamma, t)$.

\begin{proposition} \label{Prop:lowerG}
Let $\eps,\delta,\gamma$ be as in Lemma \ref{lem:lowerG}.
Then there are $M,T_1,L\in \RR^+$ such that for $t \geq T_1$,
\[
\Lambda(\GGG^M,2\gamma, t) \geq e^{-L} e^{tP(\phi)}.
\]
As a consequence, $\Lambda(X,2\gamma, t) \geq e^{-L} e^{tP(\phi)}$ for $t \geq T_1$.
\end{proposition}
\begin{proof}
We apply Lemma \ref{lem:lowerG} with $\alpha=1/2$ to obtain $M, T_1$ so that
\[
\Lambda(\GGG^M,2\gamma,2\gamma,t) \geq \tfrac{1}{2}e^{tP(\phi)}
\]
for $t \geq T_1$. Now, since $\GGG$ has the Bowen property at scale $2\gamma$, then $\GGG^M$ also has the Bowen property at scale $2\gamma$.  Letting $L := K(M) = K + 2M\Var(\phi,\eps)$ be the distortion constant from the Bowen property for $\GGG^M$, we have $\Lambda(\GGG^M,2\gamma,t) \geq e^{-L}\Lambda(\GGG^M,2\gamma,2\gamma,t)$, which gives us the required result. The elementary inequality $\Lambda(X,2\gamma,t) \geq \Lambda(\GGG^M,2\gamma,t)$ extends the result to partition sums $\Lambda(X,2\gamma,t)$.
\end{proof}

\begin{lemma}\label{lem:upperX}
Let $\eps,\delta,\gamma$ be as in Lemma \ref{lem:lowerG}.  Then there is $C_6\in \RR^+$ such that $C_6^{-1}e^{tP(\phi)} \leq \Lambda(X,2\gamma,t) \leq \Lambda(X,2\gamma,2\gamma,t) \leq C_6 e^{tP(\phi)}$ for every $t\in \RR^+$.
\end{lemma}
\begin{proof}
For the first inequality, take $T_1$ from Proposition \ref{Prop:lowerG}. Let $\kappa_1 = \inf_{0<t<T_1} \{ \Lambda(X,2\gamma,t) e^{-tP(\phi)}\}$. We have $\kappa_1>0$, since $\Lambda(X,2\gamma,t) \geq e^{\inf \phi}$. Choose $C_6$ so that $C_6^{-1}  \leq \min \{e^{-L}, \kappa\}$, and the first inequality follows from Proposition \ref{Prop:lowerG}.  The second inequality is immediate; the third comes by putting $\alpha=1/2$ in Lemma \ref{lem:lowerG} and taking $M$ accordingly, then applying Proposition \ref{prop:upper-G} to $\GGG^M$ to get for $t\geq T_1$,
\[
\Lambda(X,2\gamma,2\gamma,t) \leq 2\Lambda(\GGG^M,2\gamma,2 \gamma,t) \leq 2C_4 e^{tP(\phi)}.
\]
Let $\kappa_2 = \sup_{0<t<T_1} \{ \Lambda(X,2\gamma,t) e^{-tP(\phi)}\}$. 
Ensure that $C_6$ is chosen so that $C_6\geq \max\{2C_4, \kappa_2\}$, and the result follows.
\end{proof}

We note that the assumption that $t \geq T_1$ in Proposition \ref{Prop:lowerG} can be removed under a mild consistency condition on $\GGG^M$, which is automatically verified if $(\PPP, \GGG, \SSS)$ is a decomposition for the whole space $(X, F)$. This is the content of the following lemma.

\begin{lemma}
Let $\eps,\delta,\gamma$ be as in Lemma \ref{lem:lowerG}, and $M,T_1$ as in Proposition \ref{Prop:lowerG}. Suppose that $\GGG^M_t \neq \emptyset$ for all $0<t <T_1$. Then there exists $C_6' \in \RR^+$ so that for all $t >0$,
\[
\Lambda(\GGG^M,2\gamma, t) \geq C_6' e^{tP(\phi)}.
\]
\end{lemma}
\begin{proof}
By Proposition \ref{Prop:lowerG}, there exists $L$ so that for $t \geq T_1$, 
\[
\Lambda(\GGG^M,2\gamma, t) \geq e^{-L} e^{tP(\phi)}.
\]
Let $\kappa_3 = \inf_{0<t<T_1} \{ \Lambda(\GGG^M,2\gamma,t) e^{-tP(\phi)}\}$. Using the assumption that $\GGG^M_t \neq \emptyset$, we have $\kappa_3>0$, since $\Lambda(\GGG^M,2\gamma,t) \geq e^{\inf \phi}$. Let $C_6' = \min \{e^{-L}, \kappa_3\}$, and the result follows.
\end{proof}

\subsection{An equilibrium state with a Gibbs property}\label{sec:Gibbs}
From now on, we fix $\eps > 40\delta > 0$ so the hypotheses of Theorem \ref{thm:flowsD} are satisfied:
\begin{enumerate}
\item\label{es-spec}
for every $M\in \RR^+$ there are $T(M), \tau_M\in \RR^+$ such that $\GGG^M$ has specification at scale $\delta$ for $t \geq T(M)$ with transition time $\tau_M$; 
\item\label{es-bowen}
 $\phi$ is Bowen on $\GGG$ at scale $\eps$ (with distortion constant $K$);
\item\label{es-exp}
 $\Pexp(\phi,\eps) < P(\phi)$;
\item\label{es-gap}
 $P(\DDD^c \cup \altP \cup \altS,\delta , \eps) < P(\phi)$.
\end{enumerate}
We fix a scale $\rho \in (5\delta, \eps/8]$, and we also consider $\rho':= \rho-\delta$. For concreteness, to take scales which are all integer multiples of $\delta$, we could take:
\begin{itemize} 
\item $\eps = 48 \delta$
\item $\rho = 6 \delta$
\item $\rho' = 5 \delta$
\end{itemize}

\begin{remark} \label{rem:scales}
The scale hypotheses in \eqref{es-bowen} and \eqref{es-gap}
above can be sharpened a little. For a scale $\rho>5\delta$, where $\delta$ is the scale so that $\GGG^M$ has specification, the largest scale where we need to control expansivity occurs in \S \ref{sec:unique} and \S \ref{sec:ergodic}, where we require $\Pexp(\phi,8\rho) < P(\phi)$. The largest scale where we require the Bowen property on $\GGG$ is $3\rho$ (Lemma \ref{lem:lowerG} applied with $\gamma=\rho$). The only place we require the estimate on pressure of $\DDD^c \cup \altP \cup \altS$ is Lemma \ref{lem:many-in-G}, which is applied with $\gamma=\rho$, so it would suffice to assume that $P(\DDD^c, 2\rho, 2\rho) < P(\phi)$ and $P(\altP \cup \altS,\rho , 3 \rho) < P(\phi)$. In particular, it would suffice to assume that $P(\DDD^c \cup \altP \cup \altS,\delta) + \Var(\phi,15\delta) < P(\phi)$.
\end{remark}
The construction of an equilibrium state for $\phi$ is quite standard.  For each $t\geq 0$, let $E_t\subset X$ be a maximizing $(t,\rho')$-separated set for $\Lambda(X, \rho', t)$.  Then consider the measures
\begin{equation}\label{eqn:nutmut}
\begin{aligned}
\nu_t  &:= \frac{\sum_{x\in E_t} e^{\Phi_0(x,t)} \delta_x}{\sum_{x\in E_t} e^{\Phi_0(x,t)}}, \\
\mu_t  &:= \frac 1t \int_0^t (f_s)_*\nu_t \,ds.
\end{aligned}
\end{equation}
By compactness there is an integer valued sequence $n_k\to\infty$ such that $\mu_{n_k}$ converges in the weak* topology.  Let $\mu = \lim_k \mu_{n_k}$.

\begin{lemma}\label{lem:es}
$\mu$ is an equilibrium state for $(X,F,\phi)$.
\end{lemma}
\begin{proof} 
We compute $h_\mu(f_1)+\int \phi\, d \mu$.  Recall that $\phi^{(1)}(x) = \int_0^1 \phi(f_t x) \,dt$, and for each $n$, note that
\[
\nu_n = \frac{\sum_{x \in E_n} e^{S_n^{f_1}\phi^{(1)}(x)} \delta_x}{\sum_{x \in E_n} e^{S_n^{f_1}\phi^{(1)}(x)}}.
\]
Let $\mu^{(1)}_n = \frac{1}{n} \sum_{k=0}^{n-1}(f_1)_*^k \nu_n$, and observe that
\begin{equation}\label{eqn:mun1}
\int_0^1 (f_s)_* \mu_n^{(1)} \,ds = \frac 1n \sum_{k=0}^{n-1} \int_0^1 (f_{k+s})_* \nu_n \,ds 
= \frac 1n \int_0^n (f_s)_* \nu_n \,ds = \mu_n.
\end{equation}
Passing to a subsequence if necessary, let $\mu^{(1)} = \lim_k \mu^{(1)}_{n_k}$. The second part of the proof of \cite[Theorem 8.6]{Wa} yields
\[
h_{\mu^{(1)}}( f_1) + \int \phi^{(1)} \, d\mu^{(1)} = P(\phi^{(1)}, \rho'; f_1),
\]
where we compute $P(\phi^{(1)}, \rho'; f_1)$ in the $d_1$ metric, and thus $P(\phi^{(1)}, \rho'; f_1)=P(\phi, \rho'; F) = P(\phi; F)$, by Proposition \ref{prop:all-pressure}. Using \eqref{eqn:mun1} we get
\[
h_\mu(f_1) + \int \phi\, d\mu = h_{\mu^{(1)}}(f_1) + \int \phi^{(1)}\, d\mu^{(1)} = P(\phi; F).
\qedhere
\]
\end{proof}
The previous lemma holds whenever $P(\phi, \rho') = P(\phi)$, and thus provides the existence of an equilibrium state under this hypothesis. 
Since this result has independent interest, we state it here as a self-contained proposition.

\begin{proposition}
If $\Pexp(\phi) < P(\phi)$, then there exists an equilibrium state for $\phi$.
\end{proposition}
\begin{proof}
Since $\Pexp(\phi) < P(\phi)$, using Proposition \ref{prop:all-pressure}, we can find a scale $\gamma$ so that $P(\phi,\gamma)=P(\phi)$. We now construct $\mu$ as above but with $\gamma$ in place of $\rho'$, and carry out the argument of Lemma \ref{lem:es} to see that $\mu$ is an equilibrium state for $\phi$.
\end{proof}

The following lemma requires that $\rho\in (2\delta, \eps/8]$ and $\phi$ has the Bowen property at scale $3\rho$. This is true by our choice of $\rho$.
\begin{lemma}\label{lem:gibbs}
For sufficiently large $M$ there is $Q_M>0$ such that for every $(x,t)\in \GGG^M$ with $t\geq T(M)$ we have
\begin{equation}\label{eqn:gibbs}
\mu(B_t(x,\rho)) \geq Q_M e^{-tP(\phi) + \Phi_0(x,t)}.
\end{equation}
\end{lemma}
\begin{proof}
We apply Proposition \ref{Prop:lowerG} with $\gamma=\rho$  to obtain $M, T_1$ and $L$ so that
\[
\Lambda(\GGG^M,2\rho, t) > \tfrac 12 e^{-L} e^{tP(\phi)}.
\]
for all $t\geq T_1$.  Without loss of generality assume that $T_1 \geq T(M)$. For every $u\geq T_1$ we can find a $(u,2\rho)$-separated set $E_u' \subset \GGG_u^M$ such that
\begin{equation}\label{eqn:LambdaGM}
\sum_{x\in E_u'} e^{\Phi_0(x,u)} \geq \tfrac 12 e^{-L} e^{uP(\phi)}.
\end{equation}
Given $(x,t)\in \GGG^M$ with $t\geq T(M)$, we estimate $\mu(B_t(x,\rho))$ by estimating $\nu_s(f_{-r}B_t(x,\rho))$ for $s\gg t$ and $r\in [\tau_M+T_1, s-2\tau_M -2 T_1- t]$.  Given $s$ and $r$, let $u_1:= r-\tau_M$ and $u_2 := s-r-t-\tau_M$; 
see Figure \ref{fig:gibbs}.

\begin{figure}[htbp]
\includegraphics[width=.6\textwidth]{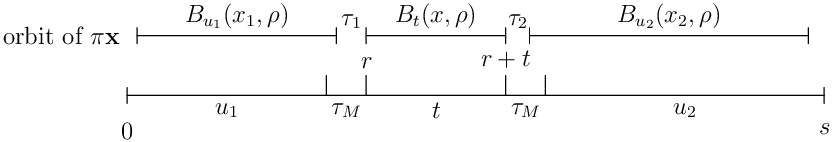}
\caption{Proving the Gibbs property.}
\label{fig:gibbs}
\end{figure}

We use similar ideas to the proof of Proposition \ref{prop:gluing} to construct a map  $\pi \colon E'_{u_1} \times E'_{u_2} \to E_t$ as follows.  By the specification property, for each $\xx \in E'_{u_1} \times E'_{u_2}$, we can find $y(\xx)$, $\tau_1(\xx), \tau_2(\xx)$ so that  
\begin{align*}
y(\xx) &\in B_{u_1}(x_1,\delta), \\
f_{u_1+\tau_1(\xx)}(y(\xx)) &\in B_t(x,\delta), \\
f_{u_1 + \tau_1(\xx) + t + \tau_2(\xx)}(y(\xx)) &\in B_{u_2}(x_2,\delta).
\end{align*}

Recall that $\rho' = \rho - \delta$, so that $\rho' > 4\delta$, and that $E_s$ denotes the maximizing $(s,\rho')$-separated set used in the construction of $\nu_s$ and $\mu_s$.  Let $\pi \colon E'_{u_1} \times E'_{u_2} \to E_s$ be given by choosing a point $\pi(\xx) \in E_s$ such that
\[
d_s(\pi(\xx), f_{\tau_1(\xx) - \tau_M}y(\xx)) \leq \rho'.
\]
For any $\xx \in E'_{u_1} \times E'_{u_2}$, we have
\[
d_t(f_r(\pi\xx), x) \leq d_t(f_r(\pi\xx),f_{r+\tau_1-\tau_M}y(\xx)) + d_t(f_{r+\tau_1-\tau_M}y(\xx), x)
< \rho' + \delta = \rho.
\]
This shows that
\begin{equation}\label{eqn:pi-in-B}
\pi\xx \in f_{-r}B_t(x,\rho).
\end{equation}
The proof of Lemma \ref{lem:multiplicity} shows there is $C_2$ such that $\#\pi^{-1}(z) \leq C_2^3$ for every $z\in E_s$.  Moreover, a mild adaptation of the proof of Lemma \ref{lem:weights} gives the existence of $C_7$ such that
\begin{equation}\label{eqn:weights-gibbs}
\Phi_0(\pi\xx,s) \geq -C_7 + \Phi_0(x_1,u_1) + \Phi_0(x_2,u_2) + \Phi_0(x,t).
\end{equation}
Now we have the estimate
\begin{equation}\label{eqn:pre-gibbs}
\begin{aligned}
\nu_s(f_{-r} B_t(x,\rho)) &= 
\frac{\sum_{z\in E_s} e^{\Phi_0(z,s)} \delta_z(f_{-r} B_t(x,\rho))}{\sum_{z\in E_s} e^{\Phi_0(z,s)}} \\
&\geq C_2^{-3} C_6^{-1} e^{-sP(\phi)} \sum_{\xx\in E'_{u_1}\times E'_{u_2}} e^{\Phi_0(\pi\xx,s)},
\end{aligned}
\end{equation}
where we use \eqref{eqn:pi-in-B} and the multiplicity bound for the estimate on the numerator, and Lemma \ref{lem:upperX} for the estimate on the denominator. (We apply the lemma with $\gamma = \rho'/2$, so we need $\rho'>4\delta$ and $\rho'\leq \eps$.)  By \eqref{eqn:weights-gibbs} and \eqref{eqn:LambdaGM} we have
\begin{align*}
\sum_{\xx\in E'_{u_1}\times E'_{u_2}} e^{\Phi_0(\pi\xx,s)}
&\geq e^{-C_7} e^{\Phi_0(x,t)} \left( \sum_{x_1\in E'_{u_1}} e^{\Phi_0(x_1,u_1)}\right)\left( \sum_{x_1\in E'_{u_2}} e^{\Phi_0(x_2,u_2)}\right) \\
&\geq \frac 14 e^{-C_7} e^{\Phi_0(x,t)} e^{-2K} e^{u_1P(\phi)}e^{u_2P(\phi)}.
\end{align*}
Together with \eqref{eqn:pre-gibbs} and the fact that $s = u_1 + u_2 + t + 2\tau_M$, this gives
\[
\nu_s(f_{-r}B_t(x,\rho)) \geq C_8 e^{\Phi_0(x,t)} e^{(u_1 + u_2 - s)P(\phi)}
= C_8 e^{-2\tau_MP(\phi)} e^{-tP(\phi) + \Phi_0(x,t)}
\]
for every $r\in [\tau_M+T_1, s-2\tau_M -2 T_1- t]$.  Integrating over $r$ gives
\begin{align*}
\mu_s(B_t(x,\rho)) &\geq \frac 1s \int_{\tau_M+T_1}^{s-2\tau_M -2 T_1- t} \nu_s(f_{-r}B_t(x,\rho)) \,dr \\
&\geq \left(1 - \frac {t+3\tau_M+3T_1}s \right) C_8 e^{-2\tau_MP(\phi)} e^{-tP(\phi) + \Phi_0(x,t)}.
\end{align*}
Sending $s\to\infty$ completes the proof of Lemma \ref{lem:gibbs}.
\end{proof}

Later we will prove that $\mu$ is ergodic.  The proof of this will use the following lemma that generalizes Lemma \ref{lem:gibbs}. We also require here that $2\rho \leq \eps$ and $\phi$ has the Bowen property at scale $3\rho$, which is true by assumption.

\begin{lemma}\label{lem:pre-ergodic}
For sufficiently large $M$ there is $Q_M'>0$ such that for every $(x_1,t_1), (x_2,t_2) \in \GGG^M$ with $t_1,t_2\geq T(M)$ and every $q\geq 2\tau_M$ there is $q' \in [q-2\tau_M, q]$ such that we have
\[
\mu(B_{t_1}(x_1,\rho) \cap f_{-(t_1 + q')}B_{t_2}(x_2,\rho)) 
\geq Q_M' e^{-(t_1 + t_2)P(\phi) + \Phi_0(x_1,t_1) + \Phi_0(x_2,t_2)}.
\]
Furthermore, we can choose $N = N(F,\delta) \in \NN$ such that $q'$ can be taken to be of the form $q-\frac{2i\tau_M}{N}$ for some $i \in \{0, \ldots N\}$.
\end{lemma}
\begin{proof}
The proof closely resembles the proof of Lemma \ref{lem:gibbs}.  We start in the same way, and, for an arbitrary $\beta>0$, and every $u\geq T_1$ we take a $(u,2\rho)$-separated set $E'_u \subset \GGG_u^M$ satisfying \eqref{eqn:LambdaGM}.
Now given $q\geq 2\tau_M+T_1$, $s\gg t_1,t_2$, and $r\geq \tau_M+T_1$, choose $u_1,u_2,u_3$ such that
\[
u_1:= r-\tau_M, \quad u_2 := q-2 \tau_M, \quad u_3 := s -r - t_1 - q - t_2 - \tau_M;
\]
see Figure \ref{fig:double-gibbs} for an illustration.

\begin{figure}[htbp]
\includegraphics[width=.95\textwidth]{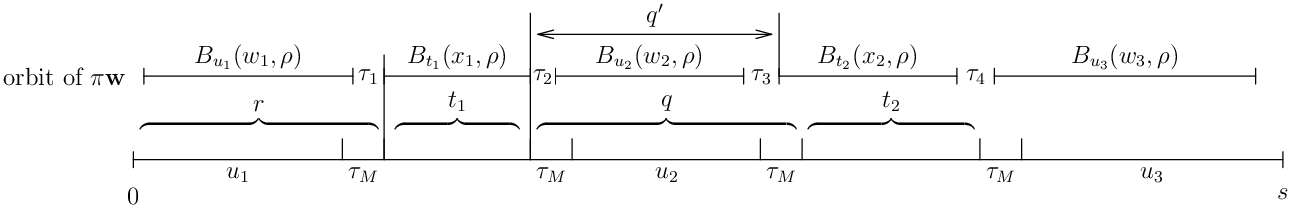}
\caption{Proving Lemma \ref{lem:pre-ergodic}.}
\label{fig:double-gibbs}
\end{figure}

We use similar ideas to the proof of Proposition \ref{prop:gluing} to construct a map  $\pi \colon E'_{u_1} \times E'_{u_2} \times E'_{u_3} \to E_s$ as follows.  By the specification property, for each $\ww \in E'_{u_1} \times E'_{u_2} \times E'_{u_3}$, we can find $y = y(\ww)$, and $ \tau_i=\tau_i(\ww)$ for $i=1,2,3,4$ so that  
\begin{align*}
y &\in B_{u_1}(w_1,\delta), \\
f_{u_1 + \tau_1}(y) &\in B_{t_1}(x_1,\delta), \\
f_{u_1 + \tau_1 + t_1 + \tau_2}(y) &\in B_{u_2}(w_2,\delta), \\
f_{u_1 + \tau_1 + t_1 + \tau_2 + u_2 + \tau_3}(y) &\in B_{t_2}(x_2,\delta), \\
f_{u_1 + \tau_1 + t_1 + \tau_2 + u_2 + \tau_3 + t_2 + \tau_4}(y) &\in B_{u_3}(w_3,\delta).
\end{align*}
Recall that $\rho' = \rho - \delta$, so that $\rho' > 4\delta$, and that $E_s$ denotes the maximizing $(s,\rho')$-separated set used in the construction of $\nu_s$ and $\mu_s$.  Let $\pi \colon E'_{u_1} \times E'_{u_2} \times E'_{u_3} \to E_s$ be given by choosing a point $\pi(\ww) \in E_t$ such that
\[
d_s(\pi(\ww), f_{\tau_1(\ww) - \tau_M}y(\ww)) \leq \rho'.
\]
In particular,
\begin{equation}\label{eqn:x1x2}
\begin{aligned}
d_{t_1}(f_r(\pi\ww), x_1) &< \rho, \\
d_{t_2}(f_{r+t_1+\hat q(\ww)}(\pi\ww), x_2) &<\rho,
\end{aligned}
\end{equation}
where we write $\hat q(\ww) := \tau_2 + u_2 + \tau_3 \in [q - 2\tau_M, q]$.  Note that similar to the proof of Lemma \ref{lem:multiplicity}, by taking $A\subset [q-2\tau_M,q]$ to be $\theta$-dense and finite for a suitable small $\theta$, we can replace $\hat q(\ww)$ with $q'(\ww)\in A$ such that \eqref{eqn:x1x2} continues to hold. We can choose $\theta=\frac{2\tau_M}{N}$ for a suitably large $N$ and set $A = \{q-i\theta: i \in\{0, \ldots, N\}\}$.  Lemmas \ref{lem:multiplicity} and \ref{lem:weights} adapt to give
\begin{align*}
\#\pi^{-1}(z) &\leq C_9, \\
\Phi_0(\pi\ww) &\geq -C_{10} + \sum_{i=1}^3 \Phi_0(w_i,u_i) + \Phi_0(x_1,t_1) + \Phi_0(x_2,t_2)
\end{align*}
Thus we have the estimate
\begin{align*}
\sum_{q'\in A} &\nu_s(f_{-r} B_{t_1}(x_1,\rho) \cap f_{-(r+t_1+q')}B_{t_2}(x_2,\rho)) \\
 &= 
\frac{\sum_{z\in E_s} e^{\Phi_0(z,s)} \delta_z(f_{-r} B_{t_1}(x_1,\rho) \cap f_{-(r+t_1+q')}B_{t_2}(x_2,\rho))}{\sum_{z\in E_s} e^{\Phi_0(z,s)}} \\
&\geq C_9^{-1} C_6^{-1} e^{-sP(\phi)} \sum_{\ww\in \prod_i E'_{u_i}} e^{\Phi_0(\pi\ww,s)} \\
&\geq (C_9C_6)^{-1} e^{-sP(\phi)} e^{-C_{10}} \sum_{\ww\in \prod_i E'_{u_i}}
e^{\sum_i \Phi_0(w_i, u_i)} e^{\Phi_0(x_1,t_1) + \Phi_0(x_2,t_2)} \\
&\geq (C_9 C_6 e^{C_{10}})^{-1}e^{(u_1 + u_2 + u_3-s)P(\phi)} e^{\Phi_0(x_1,t_1) + \Phi_0(x_2,t_2)}.
\end{align*}
Observing that $u_1 + u_2 + u_3 \in [s-(t_1 + t_2) - 4\tau_M, s]$, integrating over $r$, and sending $s\to\infty$ gives
\[
\sum_{q'\in A} \mu(B_{t_1}(x_1,\rho) \cap f_{-(t_1+q')}B_{t_2}(x_2,\rho))
\geq C_{11} e^{-(t_1+t_2)P(\phi)} e^{\Phi_0(x_1,t_1) + \Phi_0(x_2,t_2)}.
\]
Taking $Q_M' = C_{11}/\#A$, this completes the proof of Lemma \ref{lem:pre-ergodic}.
\end{proof}

\subsection{Adapted partitions and positive measure sets}\label{sec:adapted}

Recall that if $E_t$ is a $(t,2\gamma)$-separated set of maximal cardinality, then a partition $\AAA$ of $X$ is \emph{adapted} to $E_t$ if for every $w\in \AAA$ there is $x\in E_t$ such that $B_t(x,\gamma) \subset w \subset \overline B_t(x,2\gamma)$. We can write $x=x(w) \in E_t$ for the (unique) point corresponding to the partition element $w$. Conversely, given $x\in E_t$, let $w_x\in \AAA_t$ be the partition element such that
\[
B_t(x,\gamma) \subset w_x \subset \overline B_t(x,2\gamma).
\]
We now state a crucial lemma, which gives us a growth rate on partition sums for sets with positive measure for an equilibrium measure. Note that this is another place in the proof where we are forced to work with the non-standard partition sums $\Lambda(\CCC, 2\gamma, 2 \gamma, t)$. 

\begin{lemma}\label{lem:pos-for-es}
Let $\eps,\delta$ be as in the previous section, and let $\gamma \in (2\delta, \eps/8]$.
For every $\alpha\in (0,1)$, there is $C_\alpha>0$ such that the following is true.  Consider any equilibrium state  $\nu$ for $\phi$ and a family $\{E_t\}_{t>0}$ of maximizing $(t,2\gamma)$-separated sets for $\Lambda(X, 2\gamma, t)$ with adapted partitions $\AAA_t$.  If $t\in \RR^+$ and $E'_t \subset E_t$  satisfy $\nu \left( \bigcup_{x \in E'_t} w_x \right)\geq \alpha$, then letting $\CCC =\{ (x,t): x \in E'_t\}$, we have 
\[
\Lambda(\CCC, 2\gamma, 2 \gamma, t) \geq C_\alpha e^{tP(\phi)}.
\]
\end{lemma}
\begin{proof}
Fix $t>0$, and let $E_t$, $\AAA_t$ and $E'_t$ be as in the statement of the lemma.  
The proof is similar to \cite[Proposition 5.4]{CT2} and \cite[Lemma 3.12]{CT3}.  Given $w\in \AAA_t$, write $\Phi(w) = \sup_{y\in w} \Phi_0(y,t)$, and note that if $x=x(w)\in E_t$ is such that $w\subset B_t(x,2\gamma)$, then we have $\Phi(w) \leq \Phi_{2\gamma}(x,t)$.  Using the fact that $\nu$ is an equilibrium state, and our assumption that $\Pexp(\phi, \eps) < P(\phi)$ with $2\gamma \leq \eps$, we can apply Lemma \ref{lem:AD} and get
\[
tP(\phi) \leq 
\nu(D) \log \sum_{x\in E'_t} e^{\Phi_{2\gamma}(x,t)}
+ \nu(D^c) \log \sum_{x\in E_t\setminus E'_t} e^{\Phi_{2\gamma}(x,t)}
+ \log 2,
\]
where $D = \bigcup_{x\in E_t'} w_x$ and $D^c = X\setminus D$.  Applying Lemma \ref{lem:upperX} gives
\[
tP(\phi) \leq \nu(D) \log\Lambda(\CCC, 2\gamma, 2\gamma, t) + (1-\nu(D)) (tP(\phi) + \log C_6) + \log 2,
\]
and rearranging we get
\[
-\log(2C_6) \leq \nu(D)\left( -\log C_6 - tP(\phi) + \log\Lambda(\CCC, 2\gamma,2\gamma, t)\right).
\]
Since $\nu(D) \geq \alpha$, we get
\[
-\log C_6 - tP(\phi) + \log \Lambda(\CCC,2\gamma,2\gamma,t)
\geq -\frac{\log (2C_6)}{\nu(D)} \geq -\frac{\log (2C_6)}{\alpha},
\]
which suffices to complete the proof of Lemma \ref{lem:pos-for-es}.
\end{proof}

\subsection{No mutually singular equilibrium measures}\label{sec:unique}

Let $\mu$ be the equilibrium state we have constructed, and suppose that there exists another equilibrium state $\nu\perp \mu$.  Let $P\subset X$ be an $F$-invariant set such that $\mu(P)=0$ and $\nu(P)=1$, and let $\AAA_t$ be adapted partitions for maximizing $(t,2\rho)$-separated sets $E_t$. To simplify notation, we use the same symbol (such as $U$) to denote both a set of partition elements ($U\subset \AAA$) and the union of those elements ($U\subset X$).  Applying Proposition \ref{prop:approximating}, there exists $U_t \subset \AAA_t$ such that $\frac 12 (\mu + \nu)(U_t\triangle P)\to 0$. 
 Note that to apply Proposition \ref{prop:approximating}, we need $\frac 12 (\mu + \nu)$ to be almost expansive at scale $4\rho$.  This is true, since we assumed that $4\rho \leq \eps$ and $\Pexp(\phi,\eps) < P(\phi)$. In particular, we have $\nu(U_t)\to 1$ and $\mu(U_t)\to 0$ and we can assume without loss of generality that $\inf_t \nu(U_t) > 0$, and so by Lemma \ref{lem:pos-for-es}, for $\CCC =\{ (x,t): x \in E_t \cap U_t\}$, we have 
\[
\Lambda(\CCC, 2\rho, 2 \rho, t) \geq C e^{tP(\phi)},
\]
and so by Lemma \ref{lem:many-in-G}, there exists $M$ so that 
$$\Lambda(\CCC \cap \GGG^M, 2\rho, 2 \rho, t) \geq \frac C2 e^{tP(\phi)}$$
for all sufficiently large $t$.
In other words, letting $E_t^M = \{x\in E_t \mid (x,t)\in \GGG^M\}$, we have
\[
\sum_{x\in E_t^M \cap U_t} e^{\Phi_{2\rho}(x,t)} \geq \frac C2 e^{tP(\phi)},
\]
and thus by the Bowen property for $\GGG$,
\[
\sum_{x\in E_t^M \cap U_t} e^{\Phi_{0}(x,t)} \geq \frac C2e^{-K} e^{tP(\phi)}.
\]
Finally, we use the Gibbs property in Lemma \ref{lem:gibbs} to observe that since $U_t \supset B_t(x,\rho)$ for every $x\in E_t\cap U_t$, we have
\[
\mu(U_t) \geq \sum_{x\in E_t^M \cap U_t} Q_M e^{-tP(\phi) + \Phi_0(x,t)}
\geq Q_M e^{-K}C/2 > 0,
\]
contradicting the fact that $\mu(U_t)\to 0$.  This contradiction implies that any equilibrium state $\nu$ is absolutely continuous with respect to $\mu$.

\subsection{Ergodicity}\label{sec:ergodic}
The following result is the final ingredient needed to complete the proof of Theorem \ref{thm:flowsD}.
\begin{proposition}\label{prop:part-mixing}
The equilibrium state $\mu$ constructed above is ergodic.
\end{proposition}
\begin{proof}
Let $P,Q\subset X$ be measurable $F$-invariant sets. We show that  if $P$ and $Q$ have positive $\mu$-measure, then 
$$\mu(P\cap Q) > 0.$$
We achieve this by showing there exists $s\in \RR$ so that $\mu(P\cap f_{-s}Q) >0$, noting the invariance of $Q$.
The existence of an F-invariant set $P$ with $0 < \mu(P) < 1$ would contradict this inequality (letting $Q = X\setminus P$), and thus we have ergodicity.

Let $\AAA_t$ be adapted partitions for maximizing $(t,2\rho)$-separated sets $E_t$. Applying Proposition \ref{prop:approximating}, there are sets $U_t \subset \AAA_t$ such that $\mu(U_t\triangle P) \to 0$, and  $V_t \subset \AAA_t$ such that $\mu(V_t\triangle Q) \to 0$. Let $2\alpha_1:= \min\{ \mu(P), \mu(Q) \}$. We can ensure that $U_t$ and $V_t$ are chosen so that for all $t$ we have $\mu(U_t) \geq \alpha_1$, and $\mu(V_t) \geq \alpha_1$,  so by Lemma \ref{lem:pos-for-es}, putting $C=C_{\alpha_1}$,  for $\CCC^U =\{ (x,t): x \in E_t \cap U_t\}$ and $\CCC^V=\{ (x,t): x \in E_t \cap V_t\}$, we have 
$$\Lambda(\CCC^U, 2\rho, 2 \rho, t) \geq C e^{tP(\phi)}, \hspace{5pt} \Lambda(\CCC^V, 2\rho, 2 \rho, t) \geq C e^{tP(\phi)}$$
and so by Lemma \ref{lem:many-in-G}, there exists $M$ so that 
$$\Lambda(\CCC^U \cap \GGG^M, 2\rho, 2 \rho, t) \geq \frac C2 e^{tP(\phi)} \text{ and } \Lambda(\CCC^V \cap \GGG^M, 2\rho, 2 \rho, t) \geq \frac C2 e^{tP(\phi)}$$
for all sufficiently large $t$.
In other words, letting $E_{t}^M = \{x\in E_t \mid (x,t)\in \GGG^M\}$, we have
\[
\sum_{x\in E^M_t\cap U_t} e^{\Phi_{2\rho}(x,t)} \geq \frac C2 e^{tP(\phi)}
\text{ and }
\sum_{x\in E^M_t\cap V_t} e^{\Phi_{2\rho}(x,t)} \geq \frac C2 e^{tP(\phi)}.
\]
Fix $q>2\tau_M$. Note that for every $x\in E^M_{{t}}\cap U_{{t}}$, $y\in E^M_{t}\cap V_{t}$, we have
\begin{equation}\label{eqn:inUV}
B_{t}(x,\rho) \cap f_{-({t}+q)}B_{t}(y,\rho) \subset U_{t} \cap f_{-({t}+q)}V_{t}.
\end{equation}
By Lemma \ref{lem:pre-ergodic}, there exists $N$ so for each such $x,y$ there exists $q'\in [q-2\tau_M,q]$ with $q' = q-\frac{2i\tau_M}{N}$ for some $i \in \{0, \ldots N\}$ so that
\[
\mu(B_t(x,\rho) \cap f_{-(t+q')}B_t(y,\rho)) \geq Q_M' e^{-2tP(\phi) + \Phi_0(x,t) + \Phi_0(y,t)}.
\]
Summing over $x\in E^M_t\cap U_t$ and $y\in E^M_t\cap V_t$, we see that for all $t>0$, there is some value of $q' = q'(t) \in [q-2\tau_M,q]$ for which
\begin{align*}
\mu(U_t\cap f_{-(t+q')}V_t) \geq \frac {Q_M'}N \sum_{x\in E^M_t\cap U_t} \sum_{y\in E^M_t\cap V_t} e^{-2tP(\phi) + \Phi_0(x,t) + \Phi_0(y,t)} 
\geq  \frac{Q_M'}{N} \frac{C^2}{4}.
\end{align*}
Denote this final quantity by $\alpha_2$. Now choose $t_0$ large enough so that
\[
\mu(U_{t_0}\triangle P)<\tfrac{\alpha_2}{2} \text{ and } \mu(V_{t_0} \triangle Q)<\tfrac{\alpha_2}{2},
\]
and write $s =t_0 + q'(t_0)$. Note that
\[
(U_{t_0} \cap f_{-s}V_{t_0}) \setminus (P\cap f_{-s}Q) \subset (U_{t_0} \setminus P) \cup f_{-s} (V_{t_0} \setminus Q),
\]
which yields
\[
\mu(P\cap Q) > \mu(U_{t_0}\cap f_{-s}V_{t_0}) - \alpha_2 \geq 0,
\]
and the proposition follows.
\end{proof}

In conclusion, \S \ref{sec:unique} showed that any equilibrium state $\nu$ is absolutely continuous with respect to $\mu$, and since $\mu$ is ergodic, this in turn implies that $\nu=\mu$, which completes the proof of uniqueness in Theorem \ref{thm:flowsD}.

\subsection{Upper Gibbs bound}\label{sec:upper-Gibbs}

We establish a weak upper Gibbs bound on $\mu$.  We prove a version that applies globally which involves the term $\Phi_\gamma(x,t)$, rather than the term $\Phi_0(x,t)$ that appears in the usual version of the Gibbs property.  This yields a uniform upper bound in terms of $\Phi_0(x,t)$ on each $\GGG^M$. First, we require a slightly different characterization of $\mu$.
\begin{lemma} \label{mukconverges}
Suppose the hypothesis of Theorem \ref{thm:flowsD} apply, and let $\eta \leq \eps/2$. let $E'_t\subset X$ be a maximizing $(t,\eta)$-separated set for $\Lambda(X, \eta, t)$.  Then consider the measures
\begin{equation}\label{eqn:nutmut0}
\begin{aligned}
\nu'_t  := \frac{\sum_{x\in E_t'} e^{\Phi_0(x,t)} \delta_x}{\sum_{x\in E_t'} e^{\Phi_0(x,t)}}, \hspace{10pt} \mu'_t  := \frac 1t \int_0^t (f_s)_*\nu_t' \,ds.
\end{aligned}
\end{equation}
Then $\mu'_t$ converges in the weak$^{\ast}$ topology to the unique equilibrium measure $\mu$ provided by the conclusion of Theorem \ref{thm:flowsD}.
\end{lemma}
\begin{proof}
Let $\mu_{t_k}' \to \nu$ for some $t_k\to\infty$. The proof of Lemma \ref{lem:es} applies with $\rho'$ replaced by $\eta$, and shows that $\nu$ is an equilibrium measure for $\ph$; the only place in that proof where the scale is used is in the application of Proposition \ref{prop:all-pressure}, which requires $\eta \leq \eps/2$.  Since the equilibrium state for $\ph$ is unique by Theorem \ref{thm:flowsD}, then $\nu = \mu$, so $\mu$ is the only limit point of $\{\mu_t'\}_{t\in \RR^+}$, whence $\mu_t' \to \mu$.
\end{proof}

Now we establish the upper Gibbs property for $\mu$.

\begin{proposition}\label{prop:upper-Gibbs}
Suppose the hypotheses of Theorem \ref{thm:flowsD} apply
and let $\gamma \in (4\delta, \eps/4]$.  Then there is $Q>0$ such that for every $(x,t)\in X\times \RR^+$, the unique equilibrium state $\mu$ satisfies $$\mu(B_t(x,\gamma)) \leq Q e^{-tP(\phi) + \Phi_\gamma(x,t)}.$$ 
Furthermore, for each $M\in \RR^+$, there exists $Q(M)>0$ so that for every $(x,t) \in \GGG^M$, $$\mu(B_t(x,\gamma)) \leq Q(M) e^{-tP(\phi) + \Phi_0(x,t)}.$$
\end{proposition}
\begin{proof}
Let $\eta = 2\gamma$ and note that $\eta \leq \eps/2$.  For each $s\in \RR^+$, let $E_s'$ be a maximizing $(t,\eta)$-separated set for $\Lambda(X,\eta,t)$; then let $\nu'_s,\mu'_s$ be as in \eqref{eqn:nutmut0}.  Then $\mu = \lim_k \mu'_{k}$ by Lemma \ref{mukconverges}. To obtain the upper Gibbs bound  we want to estimate $\nu'_s(f_{-r}B_t(x,\gamma))$ with $s\gg t$ and $r\in [0,s-t]$.  As in the proof of Lemma \ref{lem:leqproduct}, let $E_r''$ be a maximizing $(r,\gamma)$-separated set and consider the map $\pi\colon E_s' \cap f_{-r}B_t(x,\gamma) \to E_r'' \times E_{s-t-r}''$ given by $d_r(\pi(x)_1, x) < \gamma$ and $d_{s-t-r}(\pi(x)_2,f_{r+t}x) < \gamma$.  Injectivity of $\pi$ follows here just as it did there.  Observing that $\gamma/2 \in (2\delta, \eps/8]$, we see from Lemma \ref{lem:upperX} that
\[
C_6^{-1} e^{\tau P(\phi)} \leq \Lambda(X,\gamma,\tau ) \leq \Lambda(X,\gamma,\gamma,\tau) \leq C_6 e^{\tau P(\phi)}
\]
for all $\tau\in \RR^+$, and thus
\begin{align*}
\nu'_s(f_{-r}B_t(x,\gamma)) &= \frac{\sum_{x\in E_s' \cap f_{-r}B_t(x,\gamma)} e^{\Phi_0(x,t)}}{\sum_{x\in E_s'} e^{\Phi_0(x,t)}} \\
&\leq C_6 e^{-sP(\phi)} \sum_{z_1\in E_r''} \sum_{z_2\in E_{s-r-t}''} e^{\Phi_\gamma(z_1,r)} e^{\Phi_\gamma(x,t)} e^{\Phi_\gamma(z_2,s-r-t)} \\
&\leq C_6 e^{-sP(\phi)} \Lambda(X,\gamma,\gamma,r) \Lambda(X,\gamma,\gamma,s-r-t) e^{\Phi_\gamma(x,t)} \\
&\leq (C_6)^3 e^{-sP(\phi)} e^{rP(\phi)} e^{(s-r-t)P(\phi)} e^{\Phi_\gamma(x,t)} \\
&= (C_6)^3 e^{-tP(\phi) + \Phi_\gamma(x,t)}.
\end{align*}
Averaging over $r$ and sending $s\to\infty$ completes the proof of Proposition \ref{prop:upper-Gibbs}.
\end{proof}

\section{Unique equilibrium states for homeomorphisms}\label{sec:maps}

Let $X$ be a compact metric space and $f\colon X\to X$ a homeomorphism.  Everything below can also be formulated for non-invertible continuous maps; the only difference is that then we must consider almost positive expansivity in place of almost expansivity, as described in \cite{CT3}.
\subsection{Partition sums}
Given $n\in \NN$ and $x,y\in X$ we write
\begin{equation}\label{eqn:dn}
d_n(x,y) := \sup \{ d(f^kx, f^ky) \mid 0\leq k<n\}.
\end{equation}
The \emph{Bowen ball} of order $n$ and radius $\delta$ centred at $x\in X$ is
\begin{equation}\label{eqn:Bowen-ball-map}
\begin{aligned}
B_n(x,\delta) &:= \{y\in X \mid d_n(x,y) < \delta\},\\
\overline B_n(x,\delta) &:= \{y\in X \mid d_n(x,y) \leq \delta\}.
\end{aligned}
\end{equation}
Given $\delta>0$, $n\in \NN$, and $E\subset X$, we say that $E$ is $(n,\delta)$-separated if for every distinct $x,y\in E$ we have $y\notin \overline B_n(x,\delta)$.

We view $X\times \NN$ as the space of finite orbit segments for $(X,f)$ by associating to each pair $(x,n)$ the orbit segment $\{f^k(x) \mid 0\leq k < n\}$.  Given $\CCC \subset X\times \NN$ and $n\in \NN$ we write $\CCC_n = \{x\in X \mid (x,n) \in \CCC\}$.

Now we fix a continuous potential function $\phi\colon X\to \RR$.  Given a fixed scale $\eps>0$, we use $\phi$ to assign a weight to every finite orbit segment by putting
\begin{equation}\label{eqn:Phi-map}
\Phi_\eps(x,n) = \sup_{y\in B_n(x,\eps)} \sum_{k=0}^{n-1} \phi(f^ky).
\end{equation}
In particular, $\Phi_0(x,n) = \sum_{k=0}^{n-1} \phi(f^kx)$.

Given $\CCC \subset X\times \NN$ and $n\in \NN$, we consider the \emph{partition function}
\begin{equation}\label{eqn:partition-sum-map}
\Lambda(\CCC,\phi,\delta,\eps,n) = \sup \left\{ \sum_{x\in E} e^{\Phi_{\eps}(x,n)} \mid E\subset \CCC_n \text{ is $(n,\delta)$-separated} \right\},
\end{equation}
and we write $\Lambda(\CCC,\phi,\delta,n)$ in place of $\Lambda(\CCC,\phi,\delta,0,n)$. We can restrict the supremum to sets $E \subset \CCC_n$ of maximal cardinality, and thus which are $(n, \delta)$-spanning.
When $\CCC = X\times \NN$ is the entire system, we will simply write $\Lambda(X,\phi,\delta,\eps,n)$.



The pressure of $\phi$ on $\CCC$ at scales $\delta,\eps$ is given by
\[
P(\CCC,\phi,\delta,\eps) = \ulim_{t\to\infty} \frac 1n \log \Lambda(\CCC,\phi,\delta,\eps,n).
\]
 When $\eps =0$, we simplify our notation and write $P(\CCC,\phi,\delta)$ instead of $P(\CCC,\phi,\delta,0)$.  We let
\begin{align*}
P(\CCC,\phi) &= \lim_{\delta\to 0} P(\CCC,\phi,\delta).
\end{align*}
When $\CCC = X\times \NN$ is the entire space of orbit segments we obtain the usual notion of topological pressure on the entire system, and simply write $P(\phi,\delta,\eps)$ and $P(\phi)$.

\subsection{Decompositions}
We write $\mathbb{N}_0=\mathbb{N}\cup\{0\}$.

\begin{definition}
A \emph{decomposition} $(\PPP, \GGG, \SSS)$ for $\DDD \subseteq X \times \NN$ consists of three collections $\mathcal{P}, \mathcal{G}, \mathcal{S}\subset X\times \mathbb{N}_0$ and three functions $p,g,s : \DDD \to \NN_0$ such that for every $(x,n)\in \DDD$, the values $p=p(x,n)$, $g=g(x,n)$, and $s=s(x,n)$ satisfy $n = p+g+s$, and 
\begin{equation}\label{eqn:decomposition-map}
(x,p)\in \mathcal{P}, \, (f^p(x), g)\in\mathcal{G}, \, (f^{p+g}(x), s)\in \mathcal{S}.
\end{equation}
If $\DDD = X \times \NN$, we say that $(\PPP, \GGG, \SSS)$ is a decomposition for $(X, f)$.
Given a decomposition $(\PPP,\GGG,\SSS)$ and $M\in \mathbb{N}$, we write $\GGG^M$ for the set of orbit segments $(x,n) \in \DDD$ for which 
$p \leq M$ and $s\leq M$.
\end{definition}
 
We make a standing assumption that $X \times \{0 \} \subset \PPP \cap \GGG \cap \SSS$, where each symbol $(x,0)$ is identified with the empty set. This allows for orbit segments to be decomposed in `trivial' ways; for example, $(x,n)$ can belong `purely' to one of the collections $\PPP, \GGG$ or  $\SSS$ or can transition directly from $\PPP$ to $\SSS$. This is implicit in our earlier work \cite{CT,CT2,CT3}.

\subsection{Specification}

Say that $\GGG\subset X\times \NN$ has \emph{(W)-specification at scale $\delta$} if there exists $\tau\in \NN$ such that for every $\{(x_i,n_i)\}_{i=1}^k \subset \GGG$ there exists a point $y$ and a sequence of ``gluing times'' $\tau_1,\dots,\tau_{k-1} \in \NN$ with $\tau_i \leq \tau$ such that writing $N_j = \sum_{i=1}^{j} n_i + \sum_{i=1}^{j-1}\tau_i$, and $N_0 = \tau_0 = 0$, we have
\begin{equation}\label{eqn:spec-map}
d_{n_j}(f^{N_{j-1}+\tau_{j-1}}y, x_j) < \delta \text{ for every } 1\leq j\leq k.
\end{equation}
Note that $N_j$ is the time spent for the orbit $y$ to shadow the orbit segments $(x_1, n_1)$ up to $(x_j, n_j)$. 

If it is possible to let $\tau_i = \tau$ for all ``gluing times'', then we say that $\GGG$ has \emph{(S)-specification at scale $\delta$}. The only specification property used in this paper is (W)-specification. Henceforth, when we write \emph{specification}, we mean the (W)-specification property. We also introduce tail specification in the discrete time setting.
\begin{definition}
We say that $\GGG\subset X\times \NN$ has \emph{tail specification at scale $\delta$} if there exists $N_0\in \NN$ so that $\GGG \cap (X\times [N_0,\infty))$ has specification at scale $\delta$; i.e. the specification property holds for orbit segments $(x_i, n_i) \in \GGG$ with $n_i\geq N_0$. We also sometimes write ``$\GGG$ has specification at scale $\delta$ for $n\geq N_0$'' to describe this property. 

\end{definition}

\subsection{The Bowen property} The following property agrees with the usual Bowen property in the case when $\CCC = X\times \NN$.

\begin{definition}\label{def:Bowen2}
We say that $\phi\colon X\to \RR$ has the \emph{Bowen property at scale $\eps$ on $\CCC\subset X\times \NN$} if there exists $K\in \RR$ such that $|\Phi_0(x,n) - \Phi_0(y,n)| \leq K$ for all $(x,n)\in \CCC$ and $y\in B_n(x,\eps)$.
\end{definition}
We sometimes call $K$ the \emph{distortion constant} for the Bowen property. 
Note that if $\phi$ has the Bowen property at scale $\eps$ on $\GGG$ with distortion constant K, then for any $M>0$, $\phi$ has the Bowen property at scale $\eps$ on $\GGG^M$ with distortion constant given by $K(M) = K+2M \Var(\phi, \eps)$.

\subsection{Expansivity}

Given $x\in X$ and $\eps>0$, consider the set
\begin{equation}\label{eqn:Binfty-map}
\Gamma_\eps(x) := \{y\in X \mid d(f^nx,f^ny)\leq\eps \text{ for all }n\in \ZZ\}.
\end{equation}
The map is expansive if there exists $\eps>0$ such that
\begin{equation}\label{eqn:expansive-orbit-map}
\Gamma_\eps(x) =\{x\}
\end{equation}
for every $x\in X$.  Following \cite{BF13, CT3}, we 
define the set of non-expansive points at scale $\eps$ as  $\mathrm{NE}(\eps):=\{ x\in X \mid \Gamma_\eps(x)\neq \{x\}\}$, and we say that an $f$-invariant measure $\mu$ is \emph{almost expansive at scale $\eps$} if $\mu(\mathrm{NE}(\eps))=0$.  
\begin{definition}\label{def:expansivity-map}
Given a potential $\phi$, the \emph{pressure of obstructions to expansivity at scale $\eps$} is
\begin{align*}
\Pexp(\phi, \eps) &=\sup_{\mu\in \mathcal{M}^e_f(X)}\left\{ h_{\mu}(f) + \int \phi\, d\mu \mid \mu(\mathrm{NE}(\eps))>0\right\} \\
&=\sup_{\mu\in \mathcal{M}_f^e(X)}\left\{ h_{\mu}(f) + \int \phi\, d\mu\mid \mu(\mathrm{NE}(\eps))=1\right\}.
\end{align*}
We define a scale-free quantity by
$$\Pexp(\phi) = \lim_{\eps \to 0} \Pexp(\phi, \eps).$$
\end{definition}

\subsection{Results for homeomorphisms} 

The following result is the analogue of Theorem \ref{thm:flowssimple} for homeomorphisms.

\begin{theorem}\label{thm:mapssimple}
Let $X$ be a compact metric space, $f\colon X\to X$ a homeomorphism, and $\phi\colon X\to\RR$ a continuous potential function.  Suppose that $\Pexp(\phi) < P(\phi)$ and $X\times \NN$ admits a decomposition $(\PPP, \GGG, \SSS)$ with the following properties:
\begin{enumerate}[label=\textup{\textbf{(\Roman{*})}}]
\item $\GGG$ has (W)-specification at any scale $\delta>0$;
\item $\phi$ has the Bowen property on $\GGG$;
\item $P(\PPP \cup \SSS,\phi) < P(\phi)$.
\end{enumerate}
Then $(X,f,\phi)$ has a unique equilibrium state $\mu$.  
\end{theorem}

Again, we actually prove a version of this theorem with slightly weaker hypotheses, analogous to the more general version of our theorem for flows. 
\begin{theorem}\label{thm:mapsD}
Let $X$ be a compact metric space, $f\colon X\to X$ a homeomorphism, and $\phi\colon X\to\RR$ a continuous potential function.  Suppose there are $\delta,\eps>0$ with $\eps>40 \delta$ such that $\Pexp(\phi,\eps) < P(\phi)$ and there exists $ \DDD \subset X\times \NN$ which admits a decomposition $(\PPP, \GGG, \SSS)$ with the following properties:
\begin{enumerate}[label=\textup{\textbf{(\Roman{*}$'$)}}]
\item For every $M\in \NN$, $\GGG^M$ has tail (W)-specification at scale $\delta$; 
\item $\phi$ has the Bowen property at scale $\eps$ on $\GGG$;
\item $P(\PPP \cup \SSS \cup \DDD^c,\phi, \delta, \eps)< P(\phi)$.
\end{enumerate}
Then $(X,f,\phi)$ has a unique equilibrium state.
\end{theorem}
One can also formulate versions of these results for non-invertible continuous maps.  The key difference is that in this case one should use a one-sided version of almost expansivity, as in \cite{CT3}.

The proof of Theorem \ref{thm:mapsD} follows the same outline as that of Theorem \ref{thm:flowsD}. Certain arguments simplify because some technical points associated to the flow case do not arise in discrete time. In view of this, we only provide  a sketch proof in \S\ref{sec:maps-pf}, pointing out differences with the flow case, and referring the details to corresponding arguments in \S\ref{sec:flows-pf}. The argument for showing that Theorem \ref{thm:mapssimple} is a corollary of Theorem \ref{thm:mapsD} is analogous to the flow case, using the discrete-time analogue of Lemma \ref{lem:2-specs}.

We obtain the upper level-2 large deviations principle for the equilibrium states provided by the Theorem \ref{thm:mapssimple}.
\begin{theorem} \label{thm:ldp} Suppose $\mu$ is a unique equilibrium state for $(X,f,\phi)$ provided by Theorem \ref{thm:mapssimple}. Then, $\mu$ satisfies the following upper large deviations bound: if $A$ is any weak*-closed and convex subset of the space of Borel probability measures on $X$, then
\begin{equation}\label{eqn:ldp}
\ulim_{n\to\infty} \frac 1n \log \mu (\EEE_n^{-1}(A))
\leq \sup_{\nu\in A \cap \Mf(X)} \left(h_\nu(f) +  \int\phi\,d\nu  - P(\phi) \right),
\end{equation}
where $\EEE_n(x) = \frac 1n \sum_{k=0}^{n-1} \delta_{f^k x}$ is the $n$th empirical measure associated to $x$.
\end{theorem}

The proof of Theorem \ref{thm:ldp} is given in \S\ref{sec:LDP}. 
Although we cannot derive the large deviations estimate \eqref{eqn:ldp} under the weaker hypotheses of Theorem \ref{thm:mapsD} in complete generality, we can do so as long as either: $\phi$ has the Bowen property globally (in particular, when $\phi =0$, so $\mu$ is the measure of maximal entropy), or $(X,f)$ is expansive (see \S \ref{sec:LDP}).

\subsection{Adapted partitions and generating} 

We need some results from \cite{CT3} (which are significantly easier to prove than their counterparts in \S\ref{sec:ap} and \S\ref{sec:approxinv}). We give the statements here formulated for homeomorphisms.  The first result appears as Proposition 2.6 of \cite{CT3}, although we note that the corresponding statement in \cite{CT3} is erroneously missing the factor of $2$. This factor of $2$ propagates through the proof, so the proof in that paper should yield a factor of $56$ in the main theorem rather than $28$. The proof in this paper, where scales are considered more carefully, brings the smallest integer valued factor back down to $41$.

\begin{proposition}\label{prop:generating}
If $\mu$ is almost expansive at scale $\eps$ and $\AAA$ is a measurable partition of $X$ such that every element of $\AAA$ is contained in $B(x,\eps/2)$ for some $x\in X$, then $\AAA$ is (two-sided) generating for $\mu$.
\end{proposition}

In particular, for an adapted partition $\AAA_n$ for an $(n, \eps)$-separated set of maximal cardinality, we have $h_\mu( f^n, \AAA_n) = h_\mu(f^n)$.  (This also follows from Theorem \ref{thm:aee}.)  The next result is an analogue of Proposition \ref{prop:all-pressure}; for $\phi=0$ it was proved \cite[Proposition 2.7]{CT3}, with the same caveat as above regarding the scales.

\begin{proposition}\label{Mprop:pressure-appeared}
If $\Pexp(\phi, \eps)< P(\phi)$, then $P(\phi, \eps/2) = P(\phi)$.
\end{proposition}

The following approximation lemma is \cite[Lemma 3.14]{CT3}.

\begin{lemma}\label{MLem:Partitionapprox}
Let $\mu$ be a finite Borel measure on a compact metric space $X$ and let $P\subset X$ be measurable.  Suppose $\AAA'_n$ is a sequence of partitions such that $\bigcap_n \AAA'_n(x) = \{x\}$ for every $x\in P$, and let $\delta>0$.  Then for all sufficiently large $n$ there exists a collection $U \subset \AAA'_n$ such that $\mu(U\symdiff P) < \delta$.
\end{lemma}
In our setting, this proposition is applied when $\AAA'_n=f^{-\lfloor n/2\rfloor } \AAA_n$, where $\AAA_n$ is an adapted partition, and $P$ is a positive measure $f$-invariant set for an almost expansive measure. By invariance of the set $P$, we can find a collection $U\subset \AAA_n$ so that $\mu(U\symdiff P) < \delta$, recovering the discrete-time analogue of Proposition \ref{prop:approximating}. 
\section{Sketch Proof of Theorem \ref{thm:mapsD}}\label{sec:maps-pf}
The proof is a simplified version of the proof for flows. We give a sketch proof which can be expanded into a full proof by using the details from \S3 and \cite{CT3}, and which highlights the differences with the flow case. 

\begin{remark} 
An alternative approach, which we do not pursue, is to use a suspension flow argument to build a flow from the homeomorphism (see \cite[\S 4]{BW72}), and derive the result as a corollary of Theorem \ref{thm:flowsD}.  
\end{remark}

\subsection{Lower bounds on $X$}\label{Msec:lower-X}
The statements and proofs of the following lemmas are almost identical to those of Lemma \ref{lem:leqproduct} and Lemma \ref{lem:X-lower}.
\begin{lemma}\label{Mlem:leqproduct}
For every $\gamma>0$, and $n_1,\dots,n_k\geq0$ we have
\begin{equation}\label{Meqn:leqproduct}
\Lambda(X,2\gamma,n_1+\cdots+n_k) \leq \prod_{j=1}^k \Lambda(X,\gamma,\gamma,n_j).
\end{equation}
\end{lemma}

\begin{lemma}\label{Mlem:X-lower}
Let $\eps$ satisfy $\Pexp(\phi,\eps) < P(\phi)$. For every $n \in \NN$ and $0<\gamma\leq \eps/4$, we have the inequality
$\Lambda(X,\gamma,\gamma,n) \geq e^{nP(\phi)}$.
\end{lemma}
The remark after Lemma \ref{lem:X-lower} about the relationship between `two-scale' pressure and the Bowen property also applies in the discrete time case.

\subsection{Upper bounds on $\GGG$}

Now suppose that $\GGG\subset X\times \NN_0$ has specification at scale $\delta>0$ for $n \geq N_0$.
\begin{proposition}\label{Mprop:gluing}
Suppose that $\GGG$ has specification at scale $\delta>0$ for $n \geq N_0$ with maximum gap size $\tau$, and $\zeta>\delta$ is such that $\phi$ has the Bowen property on $\GGG$ at scale $\zeta$.  
Then, for every $\gamma > 2\zeta$,  there is a constant $C_1>0$ so that for every $n_1,\dots,n_k\geq N_0$, writing  $N := \sum_{i=1}^{k} n_i + (k-1)\tau$, and $\theta := \zeta-\delta$, we have
\begin{equation}\label{eqn:gluing-map}
\prod_{j=1}^k \Lambda(\GGG,\gamma,n_j)
\leq C_1^k \Lambda(X,\theta,N).
\end{equation}
\end{proposition}
The proof is a simplification of that of Proposition \ref{Mprop:gluing}. In the discrete time case, we just show that the map $\pi$ is 1-1 on each possible configuration of gluing times. With no significant changes to the proofs from the flow case, we also obtain the following results (see Corollary \ref{rmk:gluing} and Proposition \ref{prop:upper-G}).

\begin{corollary}\label{Mrmk:gluing}
Let $\GGG, \delta, \zeta, \tau, \gamma, \phi, n_j, N, \theta, C_1$ be as in Proposition \ref{Mprop:gluing}, and let $\eta_1, \eta_2 \geq 0$. Suppose further that the Bowen property for $\phi$ holds at scale $\eta_1$. Then
\begin{equation}\label{Meqn:gluing2}
\prod_{j=1}^k \Lambda(\GGG,\theta,\eta_1,n_j)
\leq (e^{K}C_1)^k \Lambda(X,\gamma,\eta_2,N),
\end{equation}
where $K$ is the distortion constant from the Bowen property.
\end{corollary}

\begin{lemma}\label{Mlem:upper-G}
Suppose that $\GGG\subset X\times \NN$ has tail specification at scale $\delta>0$. Suppose that $\gamma > 2 \delta$,  and $\eta \geq 0$. Suppose that $\phi$ has the Bowen property on $\GGG$ at scale $\max\{\gamma/2, \eta\}$. Then there is $C_4$ such that for every $n$ we have
\begin{equation}\label{Meqn:upper-G}
\Lambda(\GGG,\gamma,\eta,n) \leq C_4 e^{nP(\phi)}.
\end{equation}
\end{lemma}

\subsection{Lower bounds on $\GGG$}

\begin{lemma}\label{Mlem:many-in-G}
Fix a scale $\gamma > 2\delta>0$.   Let $(\PPP,\GGG,\SSS)$ be a decomposition for $\DDD \subseteq X \times \NN$ such that
\begin{enumerate}
\item $\GGG$ has tail specification at scale $\delta$ ;
\item $\phi$ has the Bowen property on $\GGG$ at scale $3 \gamma$; and
\item $P(\DDD^c,\phi,2\gamma, 2\gamma) < P(\phi) \text{ and }P(\PPP \cup \SSS,\phi,\gamma , 3 \gamma) < P(\phi)$,
\end{enumerate}
Then for every $\alpha_1,\alpha_2>0$ there exists $M\in \NN$ and $N_1 \in \RR$ such that the following is true:
\begin{itemize}
\item For any $n\geq N_1$ and $\CCC\subset X\times \NN$ such that $\Lambda(\CCC,2\gamma,2\gamma, n)\geq \alpha_1 e^{nP(\phi)}$, we have $\Lambda(\CCC\cap \GGG^M,2\gamma,2\gamma,n) \geq (1-\alpha_2) \Lambda(\CCC,2\gamma,2\gamma,n)$.
\end{itemize}
\end{lemma}

This proof is like that of Lemma \ref{lem:many-in-G}, but is significantly simpler since we can partition points in $\CCC_n$ according to the integer values assigned by the decomposition.

\subsection{Consequences of lower bound on $\GGG$}
Again, throughout this section we assume that $\GGG \subset X \times \NN$ and $\delta, \eps>0$ are such that $\GGG$ has tail specification at scale $\delta$ and $\Pexp(\phi,\eps) < P(\phi)$. The following lemmas are consequences of Lemma \ref{Mlem:many-in-G}. The proofs are identical to their counterparts in \S \ref{sec:consequences}.

\begin{lemma}\label{Mlem:lowerG}
Let $\eps,\delta$ be as above and fix $\gamma\in (2\delta, \eps/8]$ such that $\phi$ has the Bowen property on $\GGG$ at scale $3\gamma$.  Then for every $\alpha>0$ there is $M\in \NN$ and $N_1 \in \NN$ such that for every $n \geq N_1$ we have
\begin{equation}\label{Meqn:lowerG}
\Lambda(\GGG^M,2\gamma,2\gamma,n) \geq (1-\alpha) \Lambda(X,2\gamma,2\gamma,n) \geq (1-\alpha)e^{nP(\phi)}.
\end{equation}
\end{lemma}

\begin{proposition} \label{MProp:lowerG}
Let $\eps,\delta,\gamma$ be as in Lemma \ref{Mlem:lowerG}. Then there are $M,N_1,L$ such that for $n \geq N_1$, 
\[
\Lambda(\GGG^M,2\gamma, n) \geq e^{-L} e^{nP(\phi)}.
\]
As a consequence, $\Lambda(X,2\gamma, t) \geq e^{-L} e^{nP(\phi)}$ for $n \geq N_1$.
\end{proposition}

\begin{lemma}\label{Mlem:upperX}
Let $\eps,\delta,\gamma$ be as in Lemma \ref{Mlem:lowerG}.  Then there is $C_6\in \RR^+$ such that $C_6^{-1}e^{nP(\phi)} \leq \Lambda(X,2\gamma,n) \leq \Lambda(X,2\gamma,2\gamma,n) \leq C_6 e^{nP(\phi)}$ for every $n\in \NN$.
\end{lemma}


\begin{lemma}
Let $\eps,\delta,\gamma, M, N_1$ be as above. Suppose that $\GGG^M_n \neq \emptyset$ for all $0<n <N_1$. Then there exists $C_6' \in \RR^+$ so that for all $n \in \NN$,
\[
\Lambda(\GGG^M,2\gamma, n) \geq C_6' e^{nP(\phi)}.
\]
\end{lemma}

\subsection{An equilibrium state with a Gibbs property}\label{Msec:Gibbs}
From now on we fix $\eps > 40\delta > 0$ so the hypotheses of Theorem \ref{thm:mapsD} are satisfied:
\begin{itemize}
\item for every $M\in \NN$ there are $N(M), \tau_M \in \NN$ such that $\GGG^M$ has specification at scale $\delta$ for $n \geq N(M)$ with transition time $\tau_M$;
\item $\phi$ is Bowen on $\GGG$ at scale $\eps$ (with constant $K$);
\item $\Pexp(\phi,\eps) < P(\phi)$;
\item $P(\DDD^c \cup \PPP \cup \SSS,\phi,\delta , \eps) < P(\phi)$.
\end{itemize}
We fix a scale $\rho\in (5\delta, \eps/8]$ and we let $\rho':= \rho-\delta$. For concreteness, we could keep in mind the values $\eps = 48 \delta, \rho = 6 \delta, \rho' = 5 \delta$. Note that the scales can be sharpened a little as described in Remark \ref{rem:scales}.

For each $n \in \NN$, let $E_n\subset X$ be a maximizing $(n,\rho')$-separated set for $\Lambda(X, n, \rho')$.  Then consider the measures
\begin{align*}
\nu_n  &:= \frac{\sum_{x\in E_n} e^{\Phi_0(x,n)} \delta_x}{\sum_{x\in E_n} e^{\Phi_0(x,n)}}, \\
\mu_n  &:= \frac 1n \sum_{i=0}^{n-1} (f^n)_*\nu_n.
\end{align*}
By compactness there is a sequence $n_k\to\infty$ such that $\mu_{n_k}$ converges in the weak* topology.  Let $\mu = \lim_k \mu_{n_k}$.
\begin{lemma}\label{lem:es-map}
$\mu$ is an equilibrium state for $(X,f,\phi)$.
\end{lemma}
This is immediate from the second part of the proof of \cite[Theorem 8.6]{Wa}, and the fact that $P(\phi) = P(\phi, \rho')$ by Lemma \ref{Mprop:pressure-appeared}. As in the flow case, this argument yields existence of an equilibrium state for $\phi$ whenever $\Pexp(\phi)<P(\phi)$.

The following two lemmas require that $8\rho \leq \eps$ and $\phi$ has the Bowen property at scale $3\rho$, which is true by assumption. The proofs are a straight adaption of those of Lemma \ref{lem:gibbs} and Lemma \ref{lem:pre-ergodic}. The proof of Lemma \ref{lem:pre-ergodic} simplifies because all transition times are integer valued, so the step that involves replacing $\hat q$ with $q'$ can be removed.

\begin{lemma}\label{Mlem:gibbs}
For sufficiently large $M$ there is $Q_M>0$ such that for every $(x,n)\in \GGG^M$ with $n\geq N(M)$ we have
\begin{equation}\label{Meqn:gibbs}
\mu(B_n(x,\rho)) \geq Q_M e^{-nP(\phi) + \Phi_0(x,n)}.
\end{equation}
\end{lemma}

\begin{lemma}\label{Mlem:pre-ergodic}
For sufficiently large $M$ there is $Q_M'>0$ such that for every $(x_1,n_1), (x_2,n_2) \in \GGG^M$ with $n_1, n_2\geq N(M)$  and every $q\geq 2\tau_M$ there is an integer $q' \in [q-2\tau_M, q]$ such that we have
\[
\mu(B_{n_1}(x_1,\rho) \cap f_{-(t_1 + q')}B_{n_2}(x_2,\rho)) 
\geq Q_M' e^{-(n_1 + n_2)P(\phi) + \Phi_0(x_1,n_1) + \Phi_0(x_2,n_2)}.
\]
\end{lemma}

\subsection{Adapted partitions and positive measure sets}\label{Msec:adapted}
The proof of the following lemma is identical to that of Lemma \ref{lem:pos-for-es}. Note that $\eps,\delta$ are as in the previous section.

\begin{lemma}\label{Mlem:pos-for-es}
Let $\nu$ be an equilibrium state for $\phi$, and let $\gamma\in (2\delta, \eps/8]$. For every $\alpha\in (0,1)$, there is $C_\alpha>0$ such that if $\{E_n\}_{n\in \NN}$ are maximizing $(n,2\gamma)$-separated sets for $\Lambda(X, 2\gamma, n)$, and $\AAA_n$ are adapted partitions for $E_n$, and $E'_n \subset E_n$  satisfies $\nu \left( \bigcup_{x \in E'_n} w_x \right)\geq \alpha$ for all $n$, then letting $\CCC =\{ (x,n): x \in E'_n\}$, we have 
$\Lambda(\CCC, 2\gamma, 2 \gamma, n) \geq C_\alpha e^{nP(\phi)}.$
\end{lemma}

\subsection{No mutually singular equilibrium measures}\label{Msec:unique}
Let $\mu$ be the equilibrium state we have constructed, and suppose that there exists another equilibrium state $\nu\perp \mu$.  Let $P\subset X$ be an $f$-invariant set such that $\mu(P)=0$ and $\nu(P)=1$, and let $\AAA_n$ be adapted partitions for maximizing $(n,2\rho)$-separated sets $E_n$. 
Applying Lemma \ref{MLem:Partitionapprox}, there exists $U_n \subset \AAA_n$ such that $\frac 12(\mu + \nu)(U_n\triangle P)\to 0$. In particular, we have $\nu(U_n)\to 1$ and $\mu(U_n)\to 0$, and we can assume without loss of generality that $\inf_n \nu(U_n) > 0$, and so by Lemma \ref{Mlem:pos-for-es}, for $\CCC =\{ (x,n): x \in E_n \cap U_n\}$, we have 
$$\Lambda(\CCC, 2\rho, 2 \rho, n) \geq C e^{nP(\phi)},$$
and so by Lemma \ref{Mlem:many-in-G}, there exists $M$ so that 
$$\Lambda(\CCC \cap \GGG^M, 2\rho, 2 \rho, n) \geq \frac C2 e^{nP(\phi)}$$
for sufficiently large $n$.
In other words, letting $E_n^M = \{x\in E_n \mid (x,n)\in \GGG^M\}$, and using the Bowen property for $\GGG$, we have
\[
\sum_{x\in E_n^M \cap U_n} e^{\Phi_{0}(x,n)} \geq e^{-K} \sum_{x\in E_n^M \cap U_n} e^{\Phi_{2\rho}(x,n)} \geq \frac C2 e^{-K} e^{nP(\phi)}.
\]
Finally, we use the Gibbs property in Lemma \ref{Mlem:gibbs} to observe that since $U_n \supset B_n(x,\rho)$ for every $x\in E_n\cap U_n$, we have
\[
\mu(U_n) \geq \sum_{x\in E_n^M \cap U_n} Q_M e^{-nP(\phi) + \Phi_0(x,n)}
\geq Q_M e^{-K}C/2 > 0,
\]
contradicting the fact that $\mu(U_n)\to 0$.  This contradiction implies that any equilibrium state $\nu$ is absolutely continuous with respect to $\mu$.

\subsection{Ergodicity}\label{Msec:ergodic}
The following result is the final ingredient needed to complete the proof of Theorem \ref{thm:mapsD}. The proof is a straight adaptation of Proposition \ref{prop:part-mixing}.
\begin{proposition}\label{Mprop:part-mixing}
The equilibrium state $\mu$ constructed above is ergodic.
\end{proposition}
In conclusion, \S \ref{sec:unique} showed that any equilibrium state $\nu$ is absolutely continuous with respect to $\mu$, and since $\mu$ is ergodic, this in turn implies that $\nu=\mu$, which completes the proof of Theorem \ref{thm:mapsD}.

\subsection{Upper Gibbs bound}\label{sec:upper-gibbs-map}

The upper Gibbs bound is proved exactly as in Proposition \ref{prop:upper-Gibbs}.

\begin{proposition}\label{prop:upper-Gibbs-map}
Let $\eps,\delta$ be as in Theorem \ref{thm:mapsD} and let $\gamma \in (4\delta, \eps/4]$.  Then there is $Q>0$ such that for every $(x,n)\in X\times \NN$, the unique equilibrium state $\mu$ satisfies $\mu(B_n(x,\gamma)) \leq Q e^{-nP(\phi) + \Phi_\gamma(x,n)}$.
\end{proposition}

\section{Proof of Theorem \ref{thm:ldp}}\label{sec:LDP}

We can obtain a large deviations result from the upper Gibbs bound in Proposition \ref{prop:upper-Gibbs-map} by applying results of Pfister and Sullivan \cite{PS05}.

Recall from \cite[Definition 3.4]{PS05} that a function $e_\mu\colon X\to \RR^+$ is said to be an \emph{upper-energy function} for $\mu$ if $-e_\mu$ is upper semicontinuous, bounded, and
\begin{equation}\label{eqn:uef}
\lim_{\eps\to 0} \ulim_{n\to\infty} \sup_{x\in X} \frac 1n \left(\log \mu(B_n(x,\eps)) + S_n e_\mu(x)\right) \leq 0,
\end{equation}
where $S_ne_\mu(x) = \sum_{i=0}^{n-1} e_\mu(f^ix)$.  If $e_\mu$ is an upper-energy function for $\mu$, then \cite[Theorem 3.2]{PS05} shows that given any closed convex subset $A$ of the space of Borel probability measures on $X$, we have
\begin{equation}\label{eqn:ldp-2}
\ulim_{n\to\infty} \frac 1n \log \mu\{ x \mid \EEE_n(x) \in A \} \leq \sup_{\nu \in A \cap \MMM_f(X)} \left(h_\nu(f) - \int e_\mu \,d\nu \right).
\end{equation}
Comparing \eqref{eqn:ldp-2} to \eqref{eqn:ldp}, we see that to prove the remaining claim of Theorem \ref{thm:ldp}, it suffices to show that $e_\mu(x) := P(\phi)-\phi(x)$ is an upper energy function for $\mu$.

To this end, observe that $-e_\mu$ is continuous and hence bounded, so it suffices to establish \eqref{eqn:uef}.  For all sufficiently small $\gamma>0$, Proposition \ref{prop:upper-Gibbs-map} gives
\begin{equation}\label{eqn:logmu}
\begin{aligned}
\log \mu(B_n(x,\gamma)) &\leq \log Q - nP(\phi) + \Phi_\gamma(x,n) \\
&\leq \log Q - nP(\phi) + S_n\phi(x) + n\Var(\phi,\gamma),
\end{aligned}
\end{equation}
and so
\[
\tfrac 1n (\log \mu(B_n(x,\gamma)) + S_n e_\mu(x))
\leq \Var(\phi,\gamma) + \tfrac 1n \log Q
\]
for every $x\in X$ and $n\in \NN$.  Sending $n\to\infty$ gives
\[
\ulim_{n\to\infty} \sup_{x\in X} \frac 1n (\log \mu(B_n(x,\gamma)) + S_n e_\mu(x)) \leq \Var(\phi,\gamma),
\]
and sending $\gamma\to 0$ gives \eqref{eqn:uef}.  By \cite[Theorem 3.2]{PS05}, this completes the proof of Theorem \ref{thm:ldp}. 

The above argument is even simpler when $\mu$ has the standard upper Gibbs property, which yields $\log \mu(B_n(x,\gamma)) \leq \log Q - nP(\phi) + S_n\phi(x)$. By Proposition \ref{prop:upper-Gibbs-map}, any unique equilibrium state $\mu$ provided by Theorem \ref{thm:mapsD} satisfies the weak upper Gibbs property
\[
\mu(B_n(x,\gamma)) \leq Q e^{-nP(\phi) + \Phi_\gamma(x,n)}.
\]
By Remark \ref{rmk:two-scale-b}, if $\phi$ has the Bowen property (in particular, if $\phi=0$), then this reduces to the standard Gibbs property, and thus $\mu$ satisfies the upper bound of the large deviations principle.  

Similarly, if $(X,f)$ is expansive at scale $\gamma$, then \eqref{eqn:logmu} gives
\[
\tfrac 1n (\log \mu(B_n(x,\gamma)) + S_n e_\mu(x))
\leq \tfrac 1n |\Phi_\gamma(x,n) - \Phi_0(x,n)| + \tfrac 1n \log Q \to 0
\]
for the unique equilibrium state $\mu$ of Theorem \ref{thm:mapsD}.
and we see once again that $\mu$ satisfies the upper bound of the large deviations principle.

\section{Proof of Theorem \ref{thm:aee} }\label{sec:aee}
We use the notation of \S \ref{sec:aee0}. We fix a homeomorphism $f$ and a metric $d$, and suppress this from our notation for this section. We begin with the following result, which is an adaptation of \cite[Proposition 2.2]{rB72}. For a set $Z \subset X$, we let
\[
\Lambda_n^{\text{span}}(Z, \delta) = \inf \left\{ \sum_{x\in E} e^{\Phi_{0}(x,n)} \mid E \text{ is $(t,\delta)$-spanning for } Z \right \}.
\]

\begin{proposition}\label{prop:coarse}
Fix $h>h^*(\mu,\eps)$.  For every $\delta,\eta>0$ there is $c\in \RR$ and $Z\subset X$ such that $\mu(Z)>1-\eta$ and
\begin{equation}\label{eqn:coarse}
\Lambda_n^{\text{span}}(B_n(x,\eps),\delta) \leq c e^{hn}
\end{equation}
for every $x\in Z$ and $n\in \NN$.
\end{proposition}
\begin{proof}
Let $Y=\{y\in X \mid h(\Gamma_\eps(x)) \leq h^*(\mu,\eps)\}$, and note that $\mu(Y)=1$.  Fix $h' \in (h^*(\mu,\eps),h)$.  Following Bowen's proof, for each $y\in Y$ choose $m(y)\in \NN$ such that there is a set $E(y)$ that $(m,\delta/4)$-spans $\Gamma_\eps(y)$ and has $\#E(y) \leq e^{h'm(y)}$.  Consider the open set
\[
U(y) = \bigcup_{z\in \Gamma_\eps(y)} B_{m(y)}(z,\delta/4).
\]
Again following Bowen's proof, note that there is $N=N(y)\in \NN$ and an open neighbourhood $V(y)$ of $y$ such that every $u\in V(y)$ has 
\[
B_{[-N, N]}(u,\eps) := \{y : d(f^iy, f^iu)< \eps \mbox{ for } i = -N, \ldots, N\} \subset U(y).
\]
Let $K\subset X$ be a $\delta/2$-spanning set, and fix $\alpha>0$ such that $h' + \alpha \log \#Q < h$.  
Let $W\subset Y$ be a compact set such that $\mu(Y')>1-\alpha/2$, and let $y_1,\dots,y_s\in Y$ be such that $W \subset \bigcup_{i=1}^s V(y_i)$.  Let
\[
N := 1 + \max \{N(y_1),\dots,N(y_s),m(y_1),\dots,m(y_s)\}.
\]
Given $n\in \NN$ and $x\in X$, let $a_n(x) = \frac 1n \#\{0\leq k < n \mid f^k(x)\notin W \}$.  For $\ell\in \NN$, let $Z_\ell = \{ x \in X \mid a_n(x)\leq \alpha \text{ for all } n\geq \ell\}$.  By ergodicity of $\mu$ and the Birkhoff ergodic theorem, there is $\ell$ such that $\mu(Z_\ell) > 1-\eta$, and we take $Z=Z_\ell$.

Given $x\in Z$ and $n\geq \max\{2N,\ell\}$, we choose integers $0=t_0 < s_1 \leq t_1 < s_2 \leq t_2<\cdots s_r \leq t_r = n$.  Put $s_1=N$ and then construct $t_i,s_i$ iteratively as follows.
\begin{enumerate}
\item $t_i$ is the smallest value of $t\geq s_i$ such that $f^t(x)\in y_{k_i}$ for some $k_i$.
\item $s_{i+1} = t_i + m(y_{k_i})$.
\end{enumerate}
Once some $s_r\geq n-N$ we put $t_r = n$.  Observe that for each $1\leq i<r$, the set $E(y_{k_i})$ is an $(m,\delta/2)$-spanning set for $U(y_{k_i}) \supset f^{t_i}(B_n(x,\eps))$.  And for each $s_i \leq t < t_i$, the point $f^t(x)$ is within $\delta/2$ of some element of $K$.  Similarly for every $0\leq t< N$ and $n-N\leq t < n$.  It follows from Bowen's Lemma 2.1 that
\[
\Lambda_n^{\text{span}}(B_n(x,\eps),\delta) \leq (\#K)^{2N + \sum_i (t_i - s_i)} e^{ h'\sum_i (s_i -t_i)}.
\]
By definition of $t_i$ and $Z$, we have $\sum_i (t_i - s_i) \leq n a_n(x)\leq \alpha n$, and so
\[
\Lambda_n^{\text{span}}(B_n(x,\eps),\delta) \leq (\#K)^{2N + \alpha n} e^{h'n}
= (\#K)^{2N} e^{n(h' + \alpha \log (\#K))} \leq (\#K)^{2N} e^{hn},
\]
where the last inequality follows from the choice of $\alpha$.  Taking $c = (\#K)^{\max\{2N,\ell\}}$ completes the proof of Proposition \ref{prop:coarse}.
\end{proof}

With $\AAA$ as above, let $\BBB$ be another partition (eventually we will send the diameter of $\BBB$ to 0), and consider the common refinement $\AAA\vee\BBB$.  Consider the conditional entropy
\begin{align*}
H_\mu(\BBB \mid \AAA) &=
\sum_{A\in \AAA} \mu(A) \sum_{B\in \BBB, B\subset A} -\mu(B)\log\mu(B), \\
h_\mu(f,\BBB \mid \AAA) &= \lim_{n\to\infty} \tfrac 1n H_\mu(\BBB \mid \AAA).
\end{align*}
A standard computation shows that $h_\mu(f, \AAA\vee\BBB) = h_\mu(f,\AAA) + h_\mu(f,\BBB\mid\AAA)$.  Moreover, if $\diam\BBB\to 0$ then $h_\mu(f,\AAA\vee\BBB) \to h_\mu(f)$; for a proof, see \cite[Lemma 3.2]{rB72}.  Thus Theorem \ref{thm:aee} will follow immediately from the following proposition.

\begin{proposition}\label{prop:conditional}
For every $\alpha>0$, there is a partition $\BBB$ such that $\diam\BBB < \alpha$ and
$h_\mu(f,\BBB\mid\AAA) \leq h^*(\mu,\eps)$.
\end{proposition}

Before proving Proposition \ref{prop:conditional} we prove the following general result, which is similar in spirit to the Katok entropy formula \cite{aK80}, but simpler to prove because it uses partitions instead of Bowen balls.  In what follows, whenever $\AAA'\subset \AAA$ for some partition $\AAA$, we continue to use our notation convention of writing $\mu(\AAA')$ to represent $\mu\left(\bigcup_{A\in \AAA'} A\right)$.

\begin{lemma}\label{lem:katok}
Let $(X,f,\mu)$ be an ergodic measure-preserving transformation and let $\AAA,\BBB$ be finite partitions of $X$.  Suppose that for some $c,\gamma>0$ and $h\geq 0$ there are sequences $\AAA_n'\subset \AAA^n$ and $\CCC_n'\subset (\AAA\vee\BBB)^n$ such that for sufficiently large $n$ we have $\mu(\AAA_n') > \gamma$,  
and that writing $\CCC_n'(A) = \{C \in \CCC_n' \mid C\in A \}$, for every $A\in \AAA_n'$ we have
\begin{enumerate}
\item 
$\mu(\CCC_n'(A)) > \gamma \mu(A)$;
\item $\#\CCC_n'(A) \leq c e^{nh}$.
\end{enumerate}
Then $h_\mu(f,\BBB\mid\AAA) \leq h$.
\end{lemma}
\begin{proof}
Since $h_\mu(f,\BBB \mid \AAA) = h_\mu(f,\AAA \vee \BBB) - h_\mu(f,\AAA)$, the Shannon--McMillan--Breiman theorem gives the following for $\mu$-a.e.\ $x\in X$:
\[
h_\mu(f,\BBB\mid\AAA) = -\lim_{n\to\infty} \frac 1n \log \left(\frac{\mu((\AAA\vee\BBB)^n(x))}{\mu(\AAA^n(x))}\right).
\]
Fix $0<\gamma' \leq \gamma^2/2$.  For every $h' < h_\mu(f,\BBB\mid\AAA)$, there is $E\subset X$ and $N\in \NN$ such that $\mu(E)\geq 1-\gamma'$ and for every $n\geq N$ and $x\in E$, we have
\begin{equation}\label{eqn:small-cylinders}
\mu\left((\AAA\vee\BBB)^n(x)\right) \leq e^{-nh'}\mu(\AAA^n(x)).
\end{equation}
Using $\mu(X\setminus E) \leq \gamma'$ and $\mu(\AAA_n')>\gamma$, 
we have 
\[
\frac{\sum_{A\in \AAA_n'} \mu(A\cap E)}{\sum_{A\in \AAA_n'} \mu(A)}
= \frac{\sum_{A\in \AAA_n'} \mu(A) - \mu(A\cap (X\setminus E))}{\sum_{A\in \AAA_n'}\mu(A)}
\geq 1-\frac{\gamma'}{\gamma} \geq 1- \frac\gamma 2.
\]
It follows that there is $A_n\in \AAA_n'$ with $\mu(A_n\cap E)/\mu(A_n) \geq 1-\gamma/2$.  For this choice of $A_n$ we have
\begin{align*}
\sum_{C\in \CCC_n'(A_n)} \mu(C\cap E) &= \mu(A_n\cap E) - \mu\left(A_n \setminus \bigcup_{C\in \CCC_n'(A_n)} C\right) \\
&\geq \left(1-\frac\gamma2\right)\mu(A_n) - (1-\gamma)\mu(A_n) = \frac\gamma2 \mu(A_n).
\end{align*}
Let $\CCC_n'' = \{C\in \CCC_n'(A_n) \mid \mu(C\cap E) > 0\}$.  Then for every $n\geq N$ and $C\in \CCC''_n$, it follows from \eqref{eqn:small-cylinders} that
\begin{equation}\label{eqn:small-cylinders2}
\mu(C) \leq e^{-nh'} \mu(A_n).
\end{equation}
Together with the estimate on $\#\CCC_n'(A_n)$ from the hypothesis of the proposition, this gives
\[
\frac \gamma2 \mu(A_n) \leq \sum_{C\in \CCC_n''} \mu(C)
\leq c e^{nh} e^{-nh'} \mu(A_n),
\]
and so we get $e^{n(h-h')} \geq \gamma/(2c)$ for all large $n$, whence $h\geq h'$.  Since this is true for any $h'< h_\mu(f,\BBB\mid\AAA)$, we conclude that $h\geq h_\mu(f,\BBB\mid\AAA)$, as desired.  This completes the proof of Lemma \ref{lem:katok}.
\end{proof}

\begin{proof}[Proof of Proposition \ref{prop:conditional}]
Let $\BBB$ be a partition of diameter $<\alpha$ whose boundary $\partial\BBB$ carries zero $\mu$-measure.

Fix $h>h^*(\mu,\eps)$ and $0<\beta'<\beta<\frac 12$.  By Stirling's formula there is a constant $K$ such that
\begin{equation}\label{eqn:stirling}
\sum_{0\leq j \leq \beta n} \binom{n}{j} \leq Ke^{(-\beta\log\beta)n}
\end{equation}
Our goal is to use this together with Proposition \ref{prop:coarse} to produce $\AAA_n',\CCC_n'$ satisfying the conditions of Lemma \ref{lem:katok} with
\begin{equation}\label{eqn:CnA}
\#\CCC_n'(A) \leq cKe^{(h-\beta\log\beta)n},
\end{equation}
and then use the fact that $h$ can be taken arbitrarily close to $h^*(\mu,\eps)$, and $\beta$ can be taken arbitrarily close to 0, to establish Theorem \ref{thm:aee}.

Our next step is to reduce to the collection of orbits that do not spend too much time near the boundary of $\BBB$.
We start by observing that because $\mu(\partial\BBB)=0$, there is $\delta>0$ such that $\mu(B(\partial\BBB,2\delta))<\beta'$, where $B(\partial\BBB,\delta))$ denotes the $\delta$-neighbourhood of the boundary of $\BBB$.  For each $n$, we consider the set
\[
\BBB'_n = \{x \mid \#\{0\leq k<n\mid d(f^kx,\partial\BBB) < 2\delta\} \leq \beta n\}.
\]
By the Birkhoff ergodic theorem we have $\mu(\BBB_n)\to 1$ as $n\to\infty$.

Fix $0<\gamma<\gamma'<1$.  By Proposition \ref{prop:coarse}, there is $Z\subset X$ such that $\mu(Z) > \gamma'$ and $c\in \RR$ such that \eqref{eqn:coarse} holds for every $x\in Z$ and $n\in \NN$.  Let $\BBB_n'' = \BBB_n \cap Z$ and note that for all sufficiently large $n$ we have $\mu(\BBB_n'')>\gamma'$.

Now let $\CCC_n' = \{C\in (\AAA\vee\BBB)^n \mid \mu(C\cap \BBB_n'')>0\}$.  As before, for each $A\in \AAA^n$ write $\CCC_n'(A) = \{C\in \CCC_n' \mid C\subset A\}$.  Let
\[
\AAA_n' = \{ A\in \AAA^n \mid \mu(\CCC_n'(A)) > \gamma \mu(A)\}.
\]
For $n$ large we have
\[
\gamma' < \mu(\BBB_n'') \leq \mu(\CCC_n') \leq \gamma (1-\mu(\AAA_n')) + \mu(\AAA_n'),
\]
and so $\mu(\AAA_n') \geq \frac{\gamma'-\gamma}{1-\gamma}$.  Thus the first two hypotheses of Lemma \ref{lem:katok} are satisfied, and it remains only to get a bound on $\#\CCC_n'(A)$ for $A\in \AAA_n'$.

To this end, fix $A\in \AAA_n'$ and $x\in \BBB_n''\cap A$.  Note that $B_n(x,\eps) \supset A = \AAA^n(x)$ since $\diam\AAA < \eps$, and since $x\in \BBB_n'' \subset Z$,  by \eqref{eqn:coarse} we see that the set $A$ admits an $(n,\delta)$-spanning set $S$ with $\#S \leq ce^{nh}$.

Now we are in a position to estimate $\#\CCC_n'(A)$.  Let $R\subset \BBB_n''$ be a set containing exactly one element of $C\cap \BBB_n''$ for every $C\in \CCC_n'(A)$.  In particular, $\#\CCC_n'(A)=\#R$, and $S$ is $(n,\delta)$-spanning for $R$.  Given $x,y\in R$, let $d_n^H(x,y)$ denote the number of times $k\in [0,n)$ for which $f^kx,f^ky$ are in different partition elements of $\BBB$.  In other words, $d_n^H(x,y)$ is the Hamming distance between the $\BBB$-codings of $x,y$.

\begin{lemma}\label{lem:hamming}
If $d_n^H(x,y) > \beta n$ then $d_n(x,y) > 2\delta$, where $d_n$ is the usual $n$th Bowen metric.
\end{lemma}
\begin{proof}
By the definition of $\BBB_n$, we see that there are at most $\beta n$ values of $k\in [0,n)$ for which $f^kx,f^ky\in B(\partial\BBB,2\delta)$.  Thus there is some $k$ for which $f^kx,f^ky$ lie in different elements of $\BBB$ and one of them is at least $2\delta$ from the boundary.
\end{proof}

Given $x\in R$, the number of $y\in R$ such that $d_n^H(x,y)\leq \beta n$ is at most $\sum_{0\leq j\leq \beta n} \binom{n}{j}$, which by \eqref{eqn:stirling} is at most $Ke^{(-\beta\log\beta)n}$.  Thus Lemma \ref{lem:hamming} implies given $x\in S$, the Bowen ball $B_n(x,\delta)$ intersects at most $Ke^{(-\beta\log\beta)n}$ elements of $R$.  In particular, we have
\[
\#R \leq Ke^{(-\beta\log\beta)n}\#S \leq cKe^{(h-\beta\log\beta)n}.
\]
This proves \eqref{eqn:CnA}, and now we can apply Lemma \ref{lem:katok} to $\AAA_n',\CCC_n'$ to get $h_\mu(f,\BBB \mid \AAA) \leq h - \beta\log\beta$.  Since $h$ can be taken arbitrarily close to $h^*(\mu,\eps)$ and $\beta$ can be taken arbitrarily close to 0, we obtain $h_\mu(f,\BBB\mid\AAA) \leq h^*(\mu,\eps)$, completing the proof of Proposition \ref{prop:conditional}.
\end{proof}

Theorem \ref{thm:aee} follows from Proposition \ref{prop:conditional} and the remarks preceding it.

\bibliographystyle{amsalpha}
\bibliography{non-uniform-ES}

\end{document}